\documentclass[reqno]{article}
\usepackage{amsfonts}
\usepackage{amsmath}
\usepackage{amsthm}
\usepackage{amssymb}
\usepackage{oldgerm}
\usepackage{fancyhdr}
\usepackage{fancybox}
\usepackage{mathrsfs}
\usepackage{epsfig}
\def\phi{\varphi }
\def\arcosh{{\rm arcosh}\>}
\swapnumbers
\theoremstyle{plain}
\newtheorem{theorem}{Theorem}[section]
\newtheorem{corollary}[theorem]{Corollary}
\newtheorem{lemma}[theorem]{Lemma}
\newtheorem{proposition}[theorem]{Proposition}
\theoremstyle{definition}

\newtheorem{remark}[theorem]{Remark}

\theoremstyle{remark}

\numberwithin{equation}{section}

\begin{document}
\title{Dispersion and  limit theorems for random walks associated with
  hypergeometric functions of type $BC$}

\author{
Michael Voit\\
Fakult\"at Mathematik, Technische Universit\"at Dortmund\\
          Vogelpothsweg 87,
           D-44221 Dortmund, Germany\\
e-mail:   michael.voit@math.uni-dortmund.de}
\date{\today}

\maketitle

\begin{abstract}
The spherical functions of the
 noncompact Grassmann manifolds $G_{p,q}(\mathbb
F)=G/K$ over the (skew-)fields
$\mathbb F=\mathbb R, \mathbb C, \mathbb H$ with rank $q\ge1$ and 
dimension parameter $p>q$ can be described
 as Heckman-Opdam hypergeometric functions of type BC, where 
the double coset space  $G//K$ is identified with the 
  Weyl chamber $ C_q^B\subset \mathbb R^q$ of type B.
The corresponding product formulas and Harish-Chandra integral representations
were recently written down  by M. R\"osler and the author in an explicit 
way such that both formulas can be extended
analytically to all real parameters  $p\in[2q-1,\infty[$, and that associated
 commutative convolution  structures $*_p$ on $C_q^B$ exist.
In this paper we study the associated moment functions and 
 the dispersion of probability measures on $C_q^B$   with the aid 
of this generalized integral representation.
This leads to  strong laws of large
numbers and  central limit theorems for associated time-homogeneous random
walks on $(C_q^B, *_p)$ where the  moment functions and the dispersion
 appear in order to determine 
 drift vectors and covariance matrices of these limit laws explicitely.
For integers $p$, all results have  interpretations for $G$-invariant
random walks on the Grassmannians $G/K$.

Besides the BC-cases we also study the spaces 
$GL(q,\mathbb F)/U(q,\mathbb F)$, which are related to Weyl chambers of type
A, and
 for which corresponding results hold. 
For the rank-one-case $q=1$, the results of this paper are well-known 
in the context of Jacobi-type hypergroups on $[0,\infty[$.
\end{abstract}



%%%%%%%%%%%%%%%%%%%%%%%%%%

\section{Introduction}

The Heckman-Opdam theory of hypergeometric functions associated 
with root systems  generalizes the classical theory of spherical functions on
Riemannian symmetric spaces; see \cite{H},
\cite{HS} and \cite{O1} for the general theory,
and \cite{NPP}, \cite{R2}, \cite{RKV}, \cite{RV1}, \cite {Sch} 
 for some recent developments.
In this paper we study these functions for the root systems of types $A$ and
$BC$ in the noncompact case. In the case $A_{q-1}$ with $q\ge2$, this theory is
connected with the groups $G:=GL(q,\mathbb F)$ with  maximal compact
subgroups $K:=U(q,\mathbb F)$ over one of the (skew-)fields 
$\mathbb F=\mathbb R, \mathbb C, \mathbb H$ with dimension
 $$d:= \dim_{\mathbb R} \mathbb F\in\{1,2,4\} \quad\quad\text{for}\quad\quad
\mathbb F=\mathbb R, \mathbb C, \mathbb H.$$ 
 Moreover, in the case
 $BC_{q}$ with $q\ge1$, these functions are related with
 the non-compact Grassmann manifolds $\mathcal G_{p,q}(\mathbb F):=G/K$
with $p> q$, where depending on $\mathbb F=\mathbb R, \mathbb C, \mathbb H$, 
the group $G$
is one of the indefinite
orthogonal, unitary or symplectic  groups
$ SO_0(q,p),\, SU(q,p)$  or $Sp(q,p)$, and $K$ is the maximal
compact subgroup  $K=
SO(q)\times SO(p), \, S(U(q)\times U(p))$ or $Sp(q)\times Sp(p),$
respectively. 

In all these group cases, we regard the $K$-spherical functions on $G$ 
(i.e., the nontrivial, $K$-biinvariant, multiplicative continuous functions 
on $G$)
as multiplicative continuous functions on the double coset space $G//K$
 where $G//K$ is equipped with the corresponding double coset convolution.
By the
$KAK$-decomposition of $G$ in the both cases above,
 the double coset space $G//K$
may be identified with the Weyl chambers
$$C_q^A:= \{t=(t_1,\cdots,t_q)\in\mathbb  R^q: \> t_1\ge t_2\ge\cdots\ge
t_q\}$$
of type $A$
and
$$C_q^B:=\{t=(t_1,\cdots,t_q)\in\mathbb  R^q: \> t_1\ge t_2\ge\cdots\ge t_q\ge0\}$$
of type $B$ respectively. In both cases, 
this identification occurs via a exponential
mapping $t\mapsto a_t\in G$ from the Weyl chamber to a system of
representatives $a_t$ of the double cosets in $G$. We now follow the notation
in
\cite{RV1} and  put
 \begin{equation}\label{a_t-A}a_t = e^{\underline t}  \end{equation}
 for $t\in C_q^A$ in
the $A$-case, 
and 
 \begin{equation}\label{a_t-BC}
a_t= \exp(H_t) =\begin{pmatrix} \cosh\underline t & \sinh \underline t & 0 \\
          \sinh  \underline t & \cosh  \underline t & 0 \\
   0 & 0 & I_{p-q}
       \end{pmatrix}  \end{equation}
for $t\in C_q^B$ in the $BC$-case respectively
 where we use the diagonal matrices
$$ e^{\underline t}:= \text{diag}(e^{t_1},\ldots,e^{t_q}), \,
  \cosh\underline t = \text{diag}(\cosh t_1,
\ldots, \cosh t_q), \, \sinh \underline t = \text{diag}(\cosh t_1,
\ldots, \cosh t_q).$$
We  use this identification of $G//K$ and the corresponding Weyl chambers
$C_q^A$ or $C_q^B$ 
from now on.

To identify the spherical functions, we   fix the  rank $q$, follow the
notation in the first part of \cite{HS}, and denote the
Heckman-Opdam hypergeometric functions associated with the root systems
\[  2\cdot A_{q-1} = \{ \pm 2(e_i-e_j): 1\leq i < j \leq
q\} \subset \mathbb R^q\]
 and
\[  2\cdot BC_q = \{ \pm 2e_i, \pm 4 e_i, \pm 2e_i \pm 2e_j: 1\leq i < j \leq
q\} \subset \mathbb R^q\]
by  $F_{A}(\lambda, k;t)$ and  $F_{BC}(\lambda, k;t)$ respectively
with spectral variable $\lambda \in \mathbb C^q$  and multiplicity parameter
$k$. The factor $2$ in the root systems originates from the known connections of
the Heckman-Opdam theory to spherical functions on symmetric spaces  in
\cite{HS} and references cited there.
 In the case $ A_{q-1} $, the spherical functions on $G//K\simeq C_q^A$ are
then given by
\[ \phi_\lambda^A(a_t)= e^{i\cdot \langle t-\pi(t),\lambda\rangle} \cdot F_{A}(i\pi(\lambda),d/2;\pi(t))
\quad\quad (t\in \mathbb R^q, \>\lambda \in \mathbb C^q  )\]
with   multiplicity
$k=d/2$   where
$$\pi:\mathbb R^q\to \mathbb R^q_0:=\{t\in\mathbb R^q:\> t_1+\ldots+t_q=0\}$$
 is the orthogonal projection w.r.t.~the standard scalar product;
 see e.g. Eq.~(6.7) of \cite{RKV}.
 In the $BC$-cases  with $p>q$, the spherical functions on
 $G//K\simeq C_q^B$ are given by
\[ \phi_\lambda^p(a_t)=F_{BC}(i\lambda,k_p;t)
\quad\quad (t\in \mathbb R^q, \>\lambda \in \mathbb C^q)\]
with three-dimensional  multiplicity
$$k_p=(d(p-q)/2, (d-1)/2, d/2)$$
corresponding to the roots $\pm 2 e_i$,  $\pm 4 e_i$ and  $2(\pm  e_i\pm
e_j)$. 

 In the $BC$-cases, the
 associated double coset convolutions $*_{p,q}$ of measures on
$C_q^B$ are written down explicitly in  \cite{R2}
 for $p\ge 2q$  such 
that these convolutions and the associated product formulas for the associated
hypergeometric functions $F_{BC}$ above
  can be extended to all real parameters $p\ge 2q-1$ by
analytic continuation where the case   $p= 2q-1$ appears 
as degenerated singular limit case.
For these continuous family of parameters
 $p\in[2q-1,\infty[$, the convolutions  $*_{p,q}$ are 
associative,
commutative, and  probability-preserving, and  they generate commutative
hypergroups $(C_q^B,*_{p,q})$ in the sense of Dunkl, Jewett, and
Spector by \cite{R2}; for the notion of hypergroups we refer to Jewett
  \cite{J}, where hypergroups were called convos, and to the
 monograph \cite{BH}.
 The results of \cite{R2} in particular imply that
 the (nontrivial) multiplicative continuous functions of these
 hypergroups $(C_q^B,*_{p,q})$
 are precisely the associated
 hypergeometric functions $t\mapsto F_{BC}(i\lambda,k_p;t)$ with $\lambda\in
 \mathbb C^q$. 

Let us now turn to a probabilistic point of view.
It is well-known from probability theory on groups that  $G$-invariant
random walks on the symmetric spaces $G/K$ as above are in a
one-to-one-correspondence with random walks on the associated double coset
hypergroups $(G//K,*)$ via the canonical projection from $G/K$ onto 
$G//K$. In this way, all limit theorems for random walks on   $(G//K,*)$ 
 admit interpretations as  limit theorems for $G$-invariant random walk
 on $G/K$.

The major aim of the present paper is to derive several limit
theorems for time-homogeneous random walks $(X_n)_{n\ge 0}$ 
on the concrete double coset hypergroups $(G//K,*)$ mentioned above as well as
on some generalizations.
For this, we shall use an analytic approach 
which allows to derive all results in the $BC$-cases
not just for the group cases  $(G//K=C_q^B,*_{p,q})$ with integers $p$,
 but also for the the intermediate cases
 $(C_q^B,*_{p,q})$ with real numbers  $p\in[2q-1,\infty[$ of R\"osler  \cite{R2}. 
In particular we present 
 strong laws of large numbers and central limit theorems
 with   $q$-dimensional  normal distributions as limits
 with explicit formulas for the
 parameters, i.e., the drift vectors and the diffusion matrices.
 In particular, the $q$-dimensional dispersion
 of probability measures on the Weyl-chambers 
$C_q^A$ and $C_q^B$ appears as drift 
 depending on the concrete underlying hypergroup convolutions. 
 For the case $BC_1$ of rank $q=1$, 
the hypergroups 
$(C_q^B,*_{p,q})$ are hypergroups on $[0,\infty[$ with Jacobi functions as
    multiplicative functions; see \cite{K} for the theory of Jacobi functions.
 These hypergroups on $[0,\infty[$ fit into
    the theory of non-compact one-dimensional Sturm-Liouville hypergroups, 
for which
our approach  is well-known; see 
\cite{Z1}, \cite{Z2}, \cite{V1}, \cite{V2}, \cite{V3}, the monograph \cite{BH},
 and papers cited there.

In order to describe the  dispersion
 and the diffusion matrices, we shall introduce
analogues of multivariate moments of
  probability measures on $C_q^A$ and $C_q^B$,
 which can be computed explicitly  via 
  so-called moment functions $m_{\bf k}:C_q^B\to\mathbb R$ for multiindices
${\bf k}=(k_1,\ldots,k_q)\in \mathbb N_0^q$ which replace the usual moment functions
 $x\mapsto x^{\bf k}:=x_1^{k_1}\cdots x_q^{k_q}$ 
on the group $(\mathbb R^q,+)$. These moment functions  $m_{\bf k}$ are
defined
as partial derivatives of the multiplicative functions $\phi_\lambda$ 
w.r.t.~the
 spectral parameters at $\lambda=-i\rho$,  where $\rho$ is the half sum of
 positive roots, and  $\phi_{-i\rho}$ is the 
 identity character $1$ of our hypergroups on $C_q^A$ or $C_q^B$.

\medskip

We recall
 that in the group cases above, 
our limit theorems on the Weyl chambers $C_q^A$ and $C_q^B$  may be 
regarded
 as limit theorems for time-homogeneous 
 group-invariant random walks on the associated symmetric spaces $G/K$ for which
the limit theorems of this paper are partially known for a long time;
see \cite{BL}, \cite{FH}, \cite{G1}, \cite{G2}, \cite{L}, \cite{Ri}, \cite{Te1}, \cite{Te2}, \cite{Tu},
\cite{Ri},
 \cite{Vi}, and references there.
On the other hand, our analytic approach goes beyond the group cases in the
$BC$-case for non-integers $p\in[2q-1,\infty[$. Moreover, we obtain  explicit
 analytic formulas for the drift vectors and diffusion matrices below in the
limit theorems which seem to be new even in the group cases. 

We  point out that we are interested in this paper mainly in the case
$BC$. As the  $A$-case in the Heckman-Opdam theory appears as a limit of the 
$BC$-case for $p\to\infty$ in some way (see \cite{RKV}, \cite{RV1} for the details),
 it is not astonishing
that all results  in the $BC$-case are also available in the  $A$-case without
additional effort. In practice,  all results
below  are proved first for the simpler  $A$-case and then
extended to the more interesting $BC$-case.

\medskip

This paper is organized as follows. For the convenience of the reader, 
 we collect all major results
 on random walks on the symmetric spaces
 $GL(q,\mathbb F)/SU(q,\mathbb F)$ and the associated Weyl chambers $C_q^A$ of
type $A$ in Section 2 without proofs. We then do the same in Section 3 for 
random walks on the
Grassmannian manifolds  $\mathcal G_{p,q}(\mathbb F)$ and the associated 
 Weyl chambers $C_q^B$ of type $B$ where in the latter case the parameter
 $p\in[2q-1,\infty[$ is continuous. The remaining sections are then devoted to
the proofs of the main results from Sections 2 and 3.
 In particular, in Section 4 we collect
some basic results from matrix analysis which are needed later. Sections 
5 and 6  contain the proofs of facts on the moment functions in the
cases A and BC respectively. There we  derive some results on
the uniform oscillatory behavior of the spherical functions and
hypergeometric functions at the spectral parameter 
$-i\rho$ which may be interesting for themselves and seem to be new even
for spherical functions.
 Sections 7 and 8 are  devoted to the
proofs of the laws of large numbers and central limit theorems.

We expect that at least parts of this paper may be
extended from the Grassmannians $G/K=G_{p,q}(\mathbb F)$ and the  chambers
$C_q^B$ to the reductive cases $U(p,q)/(U(p)\times SU(q))$ and the 
space $C_q^B\times\mathbb T$, which may be identified with the double coset
space  $U(p,q)//(U(p)\times SU(q))$, and  where 
again the spherical functions can be
described in terms of the functions $F_{BC}$; 
see Ch.~I.5 of \cite{HS} and \cite{V4}.

\section{Dispersion and limit theorems for root systems of type $A$  }

Consider the general linear group
 $G:=GL(q,\mathbb F)$ with maximal compact subgroup $K:=U(q,\mathbb F)$ with
an integer $q\ge2$ and $\mathbb F=\mathbb R, \mathbb C, \mathbb H$ as
 in the introduction.
Let
$$\sigma_{sing}(g)\in \{x=(x_1,\ldots,x_q)\in\mathbb R^q:\> x_1\ge x_2\ge\cdots\ge x_q>0\}$$
be the singular (or Lyapunov)  spectrum of $g\in G$ where
 the singular values of $g$, i.e., the  square roots of the 
eigenvalues of the positive definite matrix 
$g^*g$, are ordered by size. Using the  notation
$\ln(x_1,\ldots,x_q):=(\ln x_1,\ldots,\ln x_q)$, we  consider the
$K$-biinvariant mapping
$$\ln\sigma_{sing}: \>\> G \longrightarrow C_q^A$$
which leads to  the canonical identification of the double coset space $G//K$
with the Weyl chamber $ C_q^A$ which corresponds to the identification in
 Eq.~(\ref{a_t-A}) in the introduction.

Now consider  i.i.d.~$G$-valued random variables $(X_k)_{k\ge1}$ with the common 
$K$-biinvariant distribution $\nu_G\in M^1(G)$ and the associated 
 $G$-valued random walk $(S_k:= X_1 \cdot X_2\cdots X_k)_{k\ge0}$ with the
convention that
 $S_0$ is  the identity matrix $I_q\in G$. We now always identify the double
coset space $G//K$ with $C_q^A$ as above.
 Then, via taking the image measure of $\nu_G$
under the canonical projection from $G$ to $G//K$, the 
$K$-biinvariant distribution $\nu_G\in M^1(G)$ is in a
one-to-one-correspondence with some probability measure $\nu\in M^1(C_q^A)$.
We shall show that, under natural
moment conditions, the $C_q^A$-valued random variables 
$$\frac{ \ln\sigma_{sing}(S_k)}{k}$$
 converge a.s. to some drift vector 
$m_{\bf 1}(\nu)\in  C_q^A$, and that the distributions of $\mathbb R^q$-valued
random variables
\begin{equation}\label{normal}
\frac{1}{\sqrt k}(\ln\sigma_{sing}(S_k)-k\cdot m_{\bf 1}(\nu))
\end{equation}
tend to some normal distribution  $N(0,\Sigma^2(\nu))$ on $\mathbb R^q$. We
shall give explicit formulas for $m_{\bf 1}(\nu)$ and the covariance matrix 
$\Sigma^2(\nu)$ depending on $\nu$ and the dimension parameter $d=1,2,4$ of
$\mathbb F$.

Let us briefly compare this central limit theorem (CLT) with the existing literature.
By polar decomposition of  $g\in G$, the symmetric space $G/K$
 can be identified with the cone 
$P_q(\mathbb F)$ of positive definite  hermitian $q\times q$ matrices via
$$gK\mapsto I(g):= gg^*\in P_q(\mathbb F) \quad\quad(g\in G),$$
where $G$ acts on $P_q(\mathbb F)$ via $a\mapsto gag^*$.  In this way, we
again obtain the identification
 $G//K\simeq C_q^A$ via 
$$KgK\mapsto \ln\sigma_{sing}(g)=\frac{1}{2}\ln \sigma( gg^*)$$
where here $\sigma$ means the spectrum, i.e., the ordered eigenvalues,
 of a positive definite matrix. 
Therefore, the CLT above may be regarded as a CLT for the spectrum of
 $K$-invariant random walks on  $P_q(\mathbb F)$.
CLTs  in this context have a long history. In particular, 
 \cite{Tu}, \cite{FH}, \cite{Te1}, \cite{Te2},
\cite{Ri}, \cite{G1}, and \cite{G2} contain CLTs where, different from our CLT, 
 $\nu$ is renormalized first into some 
 measure $\nu_k\in M^1(G)$, and then the convergence of the convolution powers $\nu_k^k$ 
 is studied. Our CLT is also in principle well-known up to the explicit
 formulas for 
the drift $m_{\bf 1}(\nu)$ and the covariance matrix $\Sigma^(\nu)$;
 see  Theorem 1 of \cite{Vi},  
the  CLTs of Le Page \cite{L}, and  the part of Bougerol in the 
monograph \cite{BL}.

We now turn to the constants $m_{\bf 1}(\nu)\in  C_q^A$ and
$\Sigma^2(\nu)$. For this we follow the approach in \cite{Z1}, \cite{Z2}, \cite{V1}, and
\cite{BH}, and introduce  so-called moment functions on the double coset
hypergroups $C_q^A\simeq G//K$ via partial derivatives of the spherical
functions $\phi_\lambda^A$ w.r.t.~the spectral parameter $\lambda$ at the identity.
For this we consider the  half sum of positive roots
\begin{equation}\label{rho-a}
\rho=(\rho_1,\ldots,\rho_q)  \quad\quad\text{ with} \quad\quad
\rho_l=\frac{d}{2}(q+1-2l) \quad\quad (l=1,\ldots,q)
\end{equation}
and recapitulate the Harish-Chandra integral representation of the
 spherical functions
\begin{equation}\label{def-spherical-a}
 \phi_\lambda^A(t)= e^{i\cdot \langle t-\pi(t),\lambda\rangle} \cdot
 F_{A}(i\pi(\lambda),d/2;\pi(t)) 
\quad\quad (t\in \mathbb R^q,\> \lambda\in \mathbb C^q)
\end{equation}
from \cite{H1}, \cite{Te2}. For this we need some notations:
For  a Hermitian matrix
 $A= (a_{ij})_{i,j=1,\ldots,q}$ over $\mathbb F$ 
 we denote by $\Delta(A)$ the determinant of $A$, and by
$\Delta_r(A) = \det ((a_{ij})_{1\leq i,j\leq r})\,$ 
the $r$-th principal minor of $A$ for $r=1,\ldots,,q$.
 For $\mathbb F= \mathbb H, $
 all determinants are understood in the sense of Dieudonn\'{e}, i.e. 
$\det(A) = (\det_{\mathbb C} (A))^{1/2}$, when $A$ is considered as a complex
matrix. For any positive Hermitian $q\times q$-matrix $x$ and 
 $\lambda \in \mathbb C^q$ we now define the  power function
\begin{equation}\label{power-function}
 \Delta_\lambda(x) = \Delta_1(x)^{\lambda_1-\lambda_2} \cdot \ldots \cdot
\Delta_{q-1}(x)^{\lambda_{q-1}-\lambda_q}\cdot
\Delta_q(x)^{\lambda_{q}}.
\end{equation}
With these notations, the Harish-Chandra integral representation of the
functions in (\ref{def-spherical-a}) reads as
\begin{equation}\label{int-rep-a}
 \phi_\lambda^A(t) =\,
\int_{U(q,\mathbb F)} \Delta_{(i\lambda-\rho)/2}\bigl(u^{-1}e^{2\underline
  t}\,u\bigr) \> du  ; 
 \end{equation}
see also Section 3 of \cite{RV1} for the precise identification.
It is clear from (\ref{int-rep-a}) that  $\phi_{-i\rho}\equiv 1$,
 and that for $\lambda\in\mathbb R^n$ and $t\in C_q^A$,
$|\phi_{-i\rho+\lambda}(g)|\le1$. We mention that the set of all parameters
$\lambda\in \mathbb C^q$, for which $\phi_{\lambda}$ is bounded, is completely
 known; see \cite{R2} and \cite{NPP}.

We now follow the known approach to the dispersion for the Gelfand pairs 
$(G,K)$ (see \cite{FH}, \cite{Te1}, \cite{Te2},
\cite{Ri}, \cite{G1}, \cite{G2}) and to  moment functions on hypergroups in
 Section 7.2.2 of \cite{BH}
 (see also \cite{Z1}, \cite{Z2}, \cite{V2}, \cite{V3}): 
 For  multiindices $l=(l_1,\ldots,l_q)\in\mathbb N_0^q$ we define
the  moment functions
\begin{align}\label{moment-function-a}
m_l(t):=&
\frac{\partial^{|l|}}{\partial\lambda^l}\phi_{-i\rho-i\lambda}(t)
\Bigl|_{\lambda=0}:=\frac{\partial^{|l|}}{(\partial\lambda_1)^{l_1}\cdots(\partial\lambda_n)^{l_q}}
\phi_{-i\rho-i\lambda}(t)
\Bigl|_{\lambda=0}
\notag\\
=&\frac{1}{2^{|l|}}
\int_K (\ln\Delta_1(u^{-1}e^{2\underline t}\,u))^{l_1}\cdot
\left(\ln\left(\frac{\Delta_2(u^{-1}e^{2\underline t}\,u)}{\Delta_1(u^{-1}e^{2\underline t}\,u)}\right)\right)^{l_2}
\cdots
\left(\ln\left(\frac{\Delta_q(u^{-1}e^{2\underline t}\,u)}{\Delta_{q-1}(u^{-1}e^{2\underline t}\,u)}\right)\right)^{l_q}\> du
 \end{align}
of order $|l|:=l_1+\cdots+l_q$ for $g\in G$. Clearly, the 
last equality in (\ref{moment-function-a}) follows  from  (\ref{int-rep-a})
 by interchanging integration and  derivatives.
Using the $q$ moment functions of first order, we form the vector-valued moment function
\begin{equation}\label{m1-vector}
m_{\bf 1}(t):=(m_{(1,0,\ldots,0)}(t),\ldots,m_{(0,\ldots,0,1)}(t))
\end{equation}
of first order.
We prove  in Section 5:

\begin{proposition}\label{1.moment-function-a}
\begin{enumerate}\itemsep=-1pt
\item[\rm{(1)}]   For all $t\in C_q^A$, $m_{\bf 1}(t)\in  C_q^A$.
\item[\rm{(2)}]  There  exists a constant
$C=C(q)$ such that for all $t\in C_q^A$, 
$$\|m_{\bf 1}(t)-t\|\le C.$$
\item[\rm{(3)}] There exists a constant $C=C(q)$
such that for all  $t\in C_q^A$ and $\lambda\in\mathbb R^q$,
$$\|\phi_{-i\rho-\lambda}^A(t)- e^{i\langle \lambda,  m_{\bf 1}(t)\rangle}\|\le
C\|\lambda\|^2.$$
\end{enumerate}\end{proposition}

Similar to  the moment function $m_{\bf 1}$, we group the  moment functions of second
order by
\begin{align}\label{m2-matrix}
m_{\bf 2}(t):=&\left(\begin{array}{ccc} m_{1,1}(t)&\cdots& m_{1,q}(t)\\
\vdots &&\vdots \\ m_{q,1}(t)&\cdots& m_{q,q}(t) \end{array}\right)
\\
:=&
\left(\begin{array}{cccc} m_{(2,0,\ldots,0)}(t)&m_{(1,1,0,\ldots,0)}(t)&\cdots&m_{(1,0,\ldots,0,1)}(t)
\\ m_{(1,1,0,\ldots,0)}(t)&m_{(0,2,0,\ldots,0)}(t)&\cdots&m_{(0,1,0,\ldots,0,1)}(t)
\\ \vdots &\vdots&&\vdots
\\ m_{(1,0,\ldots,0,1)}(t)&m_{(0,1,0,\ldots,0,1)}(t)&\cdots&m_{(0,\ldots,0,2)}(t)
\end{array}\right)
\quad\quad\text{for}\quad t\in C_q^A.
\notag\end{align}
 We  derive the following facts about the $q\times q$-matrices
$\Sigma^2(t):=m_{\bf 2}(t)-m_{\bf 1}(t)^t\cdot m_{\bf 1}(t)$ 
 in Section 5:

\begin{proposition}\label{2.moment-function-a}
\begin{enumerate}\itemsep=-1pt
\item[\rm{(1)}]   For each $t\in C_q^A$, $\Sigma^2(t)$ is positive semidefinite.
\item[\rm{(2)}]  For $t=c\cdot (1,\ldots,1)\in  C_q^A$ with $c\in\mathbb R$, 
 $\Sigma^2(t)=0$. 
\item[\rm{(3)}] If $t\in C_q^A$ does not have the form of part (2), then
$\Sigma^2(t)$ has rank $q-1$.
\item[\rm{(4)}] For all $j,l=1,\ldots,q$ and $t\in C_q^A$,
$|m_{j,l}(t)|\le ((q-1)(t_1-t_q) +\max(|t_1|,|t_q|))^2 $.
\item[\rm{(5)}] 
There  exists a constant
$C=C(q)$ such that for all $t\in C_q^A$, 
$$|m_{1,1}(t) - t_1^2|\le C(|t_1|+1)
\quad\quad\text{and}\quad\quad
|m_{q,q}(t) - t_q^2|\le C(|t_q|+1)
.$$
\end{enumerate}\end{proposition}

By Proposition \ref{2.moment-function-a}(4) and (5),
  all
second moment functions $m_{j,l}$ are growing at most quadratically, and
  $m_{1,1}$ and  $m_{q,q}$ are in fact  growing quadratically.

Consider a probability measure $\nu\in M^1(C_q^A)$. We say that  $\nu$
  admits first moments if all usual first
moments $\int_{C_q^A} t_j\> d\nu(t)$ ($j=1,\ldots,q$) exist. By Proposition 
\ref{1.moment-function-a}(2) this is equivalent to require that the modified
expectation 
 $$m_{\bf 1}(\nu):=  \int_{C_q^A} m_{\bf 1}(t)\> d\nu(t) \in C_q^A\subset \mathbb R^q$$
exists. $m_{\bf 1}(\nu)$ is called dispersion of $\nu$.
In a similar way we say that  $\nu$
  admits second moments if  all usual second
moments $\int_{C_q^A} t_j^2\> d\nu(t)$ ($j=1,\ldots,q$) exist. By
Proposition 
\ref{2.moment-function-a}(4) and (5)  this means that all second
moment functions $m_{j,l}\ge0$ 
 are $\nu$-integrable. In particular, in  this
case,  also all moments of first order exist, and 
we can  form the modified symmetric $q\times q$-covariance matrix 
$$\Sigma^2(\nu):=  \int_G m_{\bf 2}\> d\nu \> 
-\> m_{\bf 1}(\nu)^t\cdot  m_{\bf 1}(\nu).$$
The rank of this positive semidefinite matrix can be determined
depending on $\nu$. This follows in a natural way from the structure of the
double coset hypergroup $G//K\simeq A_q^A$ which is the direct product of the
diagonal subgroup $D_q:=\{c\cdot (1,\ldots,1):\> c\in\mathbb R\}\subset C_q^A$ and
   the subhypergroup 
$C_q^{A,0}:=\{ t\in C_q^A: t_1+\cdots+ t_q=0\}$ which is a reduced Weyl
   chamber of type $A$. This direct product structure  explains the form 
(\ref{def-spherical-a}) of the spherical functions. It also explains 
Proposition \ref{2.moment-function-a} and the
   following result on $\Sigma^2(\nu)$: 

\begin{proposition}\label{2.moment-function-a-pos-def}
Assume that  $\nu\in M^1(C_q^A)$
  admits second moments.
\begin{enumerate}\itemsep=-1pt
\item[\rm{(1)}] If the projection of $\nu$ under the orthogonal projection
  from $ C_q^A\subset \mathbb R^q$ onto $ D_q$ is not a point measure, and if
the support of
  $\nu$ is not contained in  $D_q$, then  $\Sigma^2(\nu)$ is positive
  definite.
\item[\rm{(2)}]  If $supp\> \nu\subset D_q $, then the rank of  $\Sigma^2(\nu)$
is at most 1.
\item[\rm{(3)}] If the projection of $\nu$ under the orthogonal projection from
  $C_q^A\subset \mathbb R^q$ to $D_q$ is  a point measure, and if
  $supp\> \nu \not\subset D_q$, then  $\Sigma^2(\nu)$  has rank $q-1$.
\end{enumerate}
\end{proposition}

As main results of this paper in the A-case, we
 have the following strong law of large numbers and CLT for a biinvariant 
random walk
 $(S_k)_{k\ge0}$ on $G$ associated with the probability measure 
$\nu\in M^1(C_q^A)$. Proofs are given in Section 7 below.

\begin{theorem}\label{lln-a}
\begin{enumerate}\itemsep=-1pt
\item[\rm{(1)}]   If    $\nu$ admits  first moments, then for $k\to\infty$,
$$\frac{\ln\sigma_{sing}(S_k)}{k} \longrightarrow m_{\bf 1}(\nu)
  \quad\quad\text{almost surely.}$$
\item[\rm{(2)}]    If    $\nu$ admits second moments, then for all
  $\epsilon>1/2$ and
 $k\to\infty$,
$$\frac{1}{ k^\epsilon}\bigl(\ln\sigma_{sing}(S_k)-k\cdot m_{\bf 1}(\nu))
\longrightarrow 0
  \quad\quad\text{almost surely.}$$
\end{enumerate}
\end{theorem}

\begin{theorem}\label{clt-a}
If   $\nu\in M^1(G)$  admits finite second moments, 
then for $k\to\infty$,
$$\frac{1}{ \sqrt k}(\ln\sigma_{sing}(S_k)-k\cdot m_{\bf 1}(\nu)) \longrightarrow
N(0,\Sigma^2(\nu))  \quad \quad \quad\text{in distribution}.$$
\end{theorem}

\section{Dispersion and limit theorems for root systems of type $BC$  }

In this section we 
consider the non-compact Grassmann manifolds $\mathcal G_{p,q}(\mathbb F):=G/K$
with $p> q$, where depending on $\mathbb F$, 
the group $G$
is one of the indefinite
orthogonal, unitary or symplectic  groups
$ SO_0(q,p),\, SU(q,p)$  or $Sp(q,p)$, and $K$  the maximal
compact subgroup  $
SO(q)\times SO(p), \, S(U(q)\times U(p))$ or $Sp(q)\times Sp(p)$ respectively.
We identify the double coset space $G//K$ with the Weyl chamber
$C_q^B$ according to Eq.~(\ref{a_t-BC}). To determine the associated 
canonical projection from $G$ to $C_q^B$, 
 write $g\in G$ in $p\times q$-block notation as 
$$ g = \begin{pmatrix} A(g) & B(g)\\
C(g) & D(g)     \end{pmatrix}$$
with $A(g)\in M_q(\mathbb F)$, $D(g)\in  M_p(\mathbb F)$, and so on.
 By Eq.~(\ref{a_t-BC}), the canonical
projection from $G$ to $C_q^B$ is  given by
$$g\mapsto \arcosh(\sigma_{sing}(A(g)))$$
where $\sigma_{sing}$ again denotes the ordered singular spectrum, and 
$\arcosh$ is taken in each component. 

Similar to Section 2, we are interested in limit theorems 
for biinvariant random walks $(S_k:= X_1 \cdot X_2\cdots X_k)_{k\ge0}$ on $G$
for   i.i.d.~$G$-valued random variables $(X_k)_{k\ge1}$ with the common 
$K$-biinvariant distribution $\nu_G\in M^1(G)$.
We identify  $G//K$ with $C_q^B$ as above.
 Then, via taking the image measure of $\nu_G$
under the canonical projection from $G$ to $G//K$, the 
$K$-biinvariant distribution $\nu_G\in M^1(G)$ corresponds 
 with some unique probability measure $\nu\in M^1(C_q^B)$.
We shall show that, under natural
moment conditions, the $C_q^B$-valued random variables 
$$\frac{\arcosh(\sigma_{sing}(A(S_k)))}{k}$$
 converge a.s. to some drift vector 
$m_{\bf 1}(\nu)\in  C_q^B$, and that the  $\mathbb R^q$-valued
random variables
\begin{equation}\label{normal-B}
\frac{1}{\sqrt k}( \arcosh(\sigma_{sing}(A(S_k)))-k\cdot m_{\bf 1}(\nu))
\end{equation}
tend in distribution to some normal distribution  $N(0,\Sigma^2(\nu))$ on $\mathbb R^q$.

We  derive these limit theorems in a more general context. 
For this recall  that $C_q^B\simeq G//K$ is a double coset hypergroup
whose multiplicative functions are given by the
 hypergeometric functions
\begin{equation}\label{def-hypergeo-b}
 \phi_\lambda^p(t)=F_{BC}(i\lambda,k_p;t) 
\quad\quad(t\in C_q^B, \> \lambda\in \mathbb C^q)
\end{equation}
with   multiplicity
$k_p=(d(p-q)/2, (d-1)/2, d/2)$. In \cite{R2}, the
product formula for these spherical functions $\phi\in C(G)$, namely
$$\phi(g)\phi(h)=\int_K \phi(gkh)\> dk \quad\quad (g,h\in G),$$
was written down explicitly in terms of these
 hypergeometric functions of type BC for
 all $p\ge 2q$ as a product formula for on $G//K\simeq C_q^B$
such that this formula remains correct for  $\phi_\lambda^p$ with 
all real parameters $p\in]2q-1,\infty]$. This result from \cite{R2}
 is as follows: For all $s,t\in C_q^B$ and
 $\lambda\in \mathbb C^q$,
$$\phi_\lambda^p(t)\phi_\lambda^p(s)=
\int_{C_q^B}\phi_\lambda^p(x)\> d(\delta_s*_p\delta_t)(x)$$
where the probability measures $\delta_s*_p\delta_t\in M^1(C_q^B)$ with
compact support are given by
\begin{equation}\label{convo-formel}
(\delta_s*_p\delta_t)(f)=
\frac{1}{\kappa_p}\int_{B_q}\int_{U(q,\mathbb F)} 
f\Bigr(\arcosh(\sigma_{sing}(\sinh\underline t \,w\,\sinh\underline s \,+\,
\>\cosh \underline t\> v\> \cosh \underline s
))\Bigr)
\> dv\> dm_p(w)
\end{equation}
for functions $f\in C(C_q^B)$.
Here, $dv$ means integration w.r.t. the normalized Haar measure on
 $U(q,\mathbb F)$, $B_q$ is the matrix ball
$$B_q:= \{ w\in M_{q}(\mathbb F):\> w^*w\le I_q\},$$ 
and $dm_p(w)$ is the probability measure
\begin{equation}\label{probab-mp}
dm_p(w):=\frac{1}{\kappa_{p}} \Delta(I-w^*w)^{d(p/2+1/2-q)-1}\> dw
\quad \in M^1( B_q)
\end{equation}
where $dw$ is the Lebesgue measure on the ball $B_q$, and the normalization
constant 
$\kappa_{p}>0$ is chosen such that $dm_p(w)$ is a
probability measure. For $p=2q-1$ there is a corresponding degenerated 
formula where
then $m_p \in M^1( B_q)$ then becomes singular; see
Section 3 of \cite{R1} for  details.
By \cite{R2}, the convolution (\ref{convo-formel}) can be extended
for all
$p\in[2q-1,\infty[$
  in a unique
bilinear, weakly continuous way to a commutative and 
associative convolution $*_p$ on the Banach space of all
bounded Borel measures on $C_q^B$, such that $(C_q^B,*_p)$ becomes a 
commutative hypergroup with $0\in \mathbb R^q$ as identity.

We now use the convolution $*_p$  for $p\in [2q-1,\infty[$ and $d=1,2,4$ and 
 generalize  the Markov processes 
\begin{equation}\label{def-proj-irrfahrt}
\bigl(\tilde S_k:= \arcosh(\sigma_{sing}(A(S_k)))\bigr)_{k\ge0}
\quad\quad\text{on}\quad C_q^B
\end{equation}
  in the group cases for
 integers $p$
 as follows: 
Fix  $\nu\in M^1(C_q^B)$, and consider a time-homogeneous 
random walk $(\tilde S_k)_{k\ge0}$ on $C_q^B$ 
(associated with the parameters $p,d$)
with law $\nu$, i.e., a
time-homogeneous Markov process on  starting at the hypergroup identity
 $0\in C_q^B$ 
with transition
probability
$$P(\tilde S_{k+1}\in A|\> \tilde S_k=x)= (\delta_x *\nu)(A)
\quad\quad(x\in C_q^B, \> A\subset C_q^B \quad\text{a Borel set}).$$
By our construction, each stochastic process on $C_q^B$
 defined via Eq.~(\ref{def-proj-irrfahrt}), is in fact such a time-homogeneous 
random walk for the corresponding $p,d$.
We also point out that induction on $k$ shows easily that the distributions of
$\tilde S_k$ are given as the convolution powers $\nu^{(k)}$ w.r.t. the
convolution $*_p$.
We shall derive all limit theorems in this  setting for $p\in]2q-1,\infty[$.

To identify the data of the limit theorems, we proceed as in
Section 2  and use the Harish-Chandra integral
  representation of $\phi_\lambda^p$
in Theorem 2.4 of
 \cite{RV1}:

\begin{proposition}\label{int-rep-prop-bc}
For all  $p>2q-1$, $t\in C_q^B$, and $\lambda\in \mathbb C^q$,
\begin{equation}\label{phi-int-kurz-bc}
\phi_{\lambda}^p(t)=\int_{B_q} \int_{U(q,\mathbb F)} 
 \Delta_{(i\lambda-\rho^{BC})/2}( g(t,u,w))
\> du\> dm_p(w)
\end{equation}
with the power function $\Delta_\lambda$ from (\ref{power-function}),
the half sum of positive roots
\begin{equation}\label{rho-BC}
 \rho^{BC} =  \rho^{BC}(p) =
 \sum_{i=1}^q \bigl( \frac{d}{2}(p+q+2-2i) -1 \bigr)e_i\,,
\end{equation}
\begin{equation}\label{def-g}
 g(t,u,w):=u^*(\cosh \underline t + \sinh\underline t \cdot w)(\cosh\underline t + \sinh\underline t \cdot w)^*u,
\end{equation}
and with  $m_p(w)\in M^1(B_q)$ from (\ref{probab-mp}).
For $p=2q-1$,  a corresponding degenerated formula holds.
\end{proposition}

\begin{proof}
This formula follows immediately from Theorem 2.4 of
 \cite{RV1}. Notice that
that our function $g(t,u,w)$
is equal to the function $\tilde g_t(u,w)$ in Section 2 of \cite{RV1}.
Moreover, in \cite{RV1} we take one integral over the identity component
$U_0(q,\mathbb F)$ of $U(q,\mathbb F)$ instead over  $U(q,\mathbb F)$.
But this makes a difference for these groups for 
$\mathbb F=\mathbb R$ only, where
the integrals are equal in all cases by the form of  $g(t,u,w)$.
\end{proof}

We now proceed as in Section 2. 
  For   $l=(l_1,\ldots,l_q)\in\mathbb N_0^q$ we define
the  moment functions
\begin{align}\label{def-m1-bc}
&m_l(t):=
\frac{\partial^{|l|}}{\partial\lambda^l}\phi_{-i\rho^{BC}-i\lambda}^p(t)
\Bigl|_{\lambda=0}:=
\frac{\partial^{|l|}}{(\partial\lambda_1)^{l_1}\cdots(\partial\lambda_q)^{l_q}}
\phi_{-i\rho^{BC}-i\lambda}^p(t) \Bigl|_{\lambda=0}
\notag\\
=&
\frac{1}{2^{|l|}}\int_{B_q} \int_{U(q,\mathbb F)}
 (\ln\Delta_1( g(t,u,w)))^{l_1}\cdot
 \left(\ln \frac{\Delta_2( g(t,u,w))}{\Delta_1( g(t,u,w))}\right)^{l_2}
\cdots
 \left(\ln \frac{\Delta_q( g(t,u,w))}{\Delta_{q-1}( g(t,u,w))}\right)^{l_q}
\> du\> dm_p(w)
 \end{align}
of order $|l|$ for $t\in C_q^B$. Clearly, the 
last equality follows  from  (\ref{phi-int-kurz-bc})
 by interchanging integration and  derivatives.
Using the $q$ moment functions $m_l$ of first order with $|l|=1$, 
we form the vector-valued moment function
\begin{equation}\label{m1-vector-bc}
m_{\bf 1}(t):=(m_{(1,0,\ldots,0)}(t),\ldots,m_{(0,\ldots,0,1)}(t))
\end{equation}
of first order. We prove the following properties of $m_{\bf 1}$ in Section 6:

\begin{proposition}\label{1.moment-function-bc}
\begin{enumerate}\itemsep=-1pt
\item[\rm{(1)}]   For all $t\in C_q^B$, $m_{\bf 1}(t)\in  C_q^B$.
\item[\rm{(2)}]  There  exists a constant
$C=C(p,q)$ such that for all $t\in C_q^B$, 
$$\|m_{\bf 1}(t)-t\|\le C.$$
\item[\rm{(3)}] There exists a constant $C=C(p,q)$
such that for all  $t\in C_q^B$ and $\lambda\in\mathbb R^q$,
$$\|\phi_{-i\rho-\lambda}^A(t)- e^{i\langle \lambda,  m_{\bf 1}(t)\rangle}\|\le
C\|\lambda\|^2.$$
\end{enumerate}\end{proposition}

As in Section 2 we also form the matrix consisting of all second
order  moment functions  with
\begin{align}\label{m2-matrix-bc}
m_{\bf 2}(t):=&\left(\begin{array}{ccc} m_{1,1}(t)&\cdots& m_{1,q}(t)\\
\vdots &&\vdots \\ m_{q,1}(t)&\cdots& m_{q,q}(t) \end{array}\right)
\\
:=&
\left(\begin{array}{cccc} m_{(2,0,\ldots,0)}(t)&m_{(1,1,0,\ldots,0)}(t)&\cdots&m_{(1,0,\ldots,0,1)}(t)
\\ m_{(1,1,0,\ldots,0)}(t)&m_{(0,2,0,\ldots,0)}(t)&\cdots&m_{(0,1,0,\ldots,0,1)}(t)
\\ \vdots &\vdots&&\vdots
\\ m_{(1,0,\ldots,0,1)}(t)&m_{(0,1,0,\ldots,0,1)}(t)&\cdots&m_{(0,\ldots,0,2)}(t)
\end{array}\right)
\quad\quad\text{for}\quad t\in C_q^B.
\notag\end{align}
 By Section 6, the symmetric $q\times q$-matrices
$\Sigma^2(t):=m_{\bf 2}(t)-m_{\bf 1}(t)^t\cdot m_{\bf 1}(t)$ 
 have the following properties:

\begin{proposition}\label{2.moment-function-bc}
\begin{enumerate}\itemsep=-1pt
\item[\rm{(1)}]   For each $t\in C_q^B$, the matrix  $\Sigma^2(t)$ is positive semidefinite.
\item[\rm{(2)}]  $\Sigma^2(0)=0$. 
\item[\rm{(3)}] For $t\in C_q^B$ with $t\ne 0$, the matrix 
$\Sigma^2(t)$ has full rank $q$.
\item[\rm{(4)}] There  exists a constant
$C=C(q)$ such that for all $j,l=1,\ldots,q$ and $t\in C_q^B$,
$|m_{j,l}(t)|\le C\cdot t_1^2$.
\item[\rm{(5)}] There  exists a constant
$C=C(q)$ such that for all $t\in C_q^B$, 
$$|m_{1,1}(t) - t_1^2|\le C(|t_1|+1).$$
\end{enumerate}\end{proposition}

Parts (4), (5) yield
 that all
second moment functions $m_{j,l}$ are growing at most quadratically, and that
at least  $m_{1,1}$ is  growing quadratically.

Now consider a probability measure $\nu\in M^1(C_q^B)$. As in Section 2
we say that  $\nu$
  admits first or second moments if all components of $m_{\bf 1}$ or  $m_{\bf 2}$
are integrable w.r.t.~$\nu$ respectively. In case of existence, we form
the vector  $m_{\bf 1}(\nu)\in C_q^B$ and the matrix   $\Sigma^2(\nu)$ as in
Section 2. We then have the following result which is slightly different from
the corresponding one in the A-case in Section 2: 

\begin{proposition}\label{2.moment-function-bc-pos-def}
If  $\nu\in M^1(C_q^B)$ admits second moments, and if
$\nu\ne\delta_0$, then 
$\Sigma^2(\nu)$ has full rank $q$.
\end{proposition}

As main results of this paper in the BC-case, we
 have the following strong law of large numbers and CLT for time-homogeneous
 random walk
 $(\tilde S_k)_{k\ge0}$ on $G$ associated with the probability measure 
$\nu\in M^1(C_q^B)$ which is completely analog to the corresponding results in
 the A-case in Section 2. The proofs, which are completely analog to the  
A-case, are given in Section 8.

\begin{theorem}\label{lln-bc}
\begin{enumerate}\itemsep=-1pt
\item[\rm{(1)}]   If    $\nu$ admits  first moments, then for $k\to\infty$,
$$\frac{\tilde S_k}{k} \longrightarrow m_{\bf 1}(\nu)
  \quad\quad\text{a.s..}$$
\item[\rm{(2)}]    If    $\nu$ admits second moments, then for all
  $\epsilon>1/2$ and
 $k\to\infty$,
$$\frac{1}{ k^\epsilon}\bigl(\tilde S_k-k\cdot m_{\bf 1}(\nu))
\longrightarrow 0
  \quad\quad\text{almost surely.}$$
\end{enumerate}
\end{theorem}

\begin{theorem}\label{clt-bc}
If   $\nu\in M^1(G)$  admits finite second moments, 
then for $k\to\infty$,
$$\frac{1}{ \sqrt k}(\tilde S_k-k\cdot m_{\bf 1}(\nu)) \longrightarrow
N(0,\Sigma^2(\nu))  \quad\quad\quad\text{in distribution}.$$
\end{theorem}

\section{Some results from matrix analysis}

 It this section we collect some results from matrix analysis which are needed later. Possibly, some of these
results are well-known, but we were unable to find references.
We always assume that $\mathbb F=\mathbb R, \mathbb C, \mathbb H$ and $q\ge2$.
Moreover, $M_r(\mathbb F)$ is the vector space of all $r\times r$-matrices over
 $\mathbb F$.

We start with the following observation from linear algebra.

\begin{lemma}\label{glech-det}
Let $u\in U(q,\mathbb F)$ have the block structure
 $u=\left(\begin{array}{cc} u_1&*\\ *&u_2\end{array}\right)$
with quadratic blocks $u_1\in M_r(\mathbb F)$ and $u_2\in M_{q-r}(\mathbb F)$ 
with $1\le r\le q$. Then $|\det u_1|=|\det u_2|$.
\end{lemma}

\begin{proof} W.l.o.g. we  assume $2r\le q$. 
By the $KAK$-decomposition of  $U(q,\mathbb F)$
with $K=U(r,\mathbb F)\times U({q-r},\mathbb F)$ 
(see e.g.~Theorem VII.8.6 of \cite{H2}), we may write
$u$ as
$$u=\left(\begin{array}{cc} a_1&0\\0&b_1\end{array}\right)
\cdot\left(\begin{array}{ccc} c&s&0\\-s&c&0\\0&0&I_{q-2r}\end{array}\right)
\cdot\left(\begin{array}{cc} a_2&0\\0&b_2\end{array}\right)
$$
with $a_1,a_2\in U(r,\mathbb F)$, $b_1,b_2 \in U({q-r},\mathbb F)$ and 
$$c=diag(\cos\phi_1,\ldots,\cos\phi_r), \quad\quad s=diag(\sin\phi_1,\ldots,\sin\phi_r)$$ for suitable
$ \phi_1,\ldots,\phi_r\in\mathbb R$. Therefore, 
$$u_1=a_1ca_2
\quad\quad\text{and}\quad\quad
 u_2=b_1\left(\begin{array}{cc} c&0\\0&I_{q-2r}\end{array}\right)b_2$$
 which immediately implies the claim.
\end{proof}

We next turn to some results on the principal minors $\Delta_r$:

\begin{lemma}\label{pos-coeff}
Let $1\le r\le q$ be integers and
$u\in U(q,\mathbb F)$. Consider the polynomial
$$ h_r(a_1,\ldots, a_q):=\Delta_r(u^*\cdot diag(a_1,\ldots,a_q)\cdot u)
\quad\quad \text{for}\quad a_{1},\ldots, a_{q}\in \mathbb R.$$
Then
$$h_r(a_1,\ldots, a_q)=\sum_{1\le i_1<i_2<\ldots<i_r\le q}
 c_{i_1,\ldots,i_r} a_{i_1}\cdot a_{i_2}\cdots \cdot a_{i_r}$$
 with coefficients $c_{i_1,\ldots,i_r}\ge0$
for all $1\le i_1<i_2<\ldots<i_r\le q$ and
$\sum_{1\le i_1<i_2<\ldots<i_r\le q} c_{i_1,\ldots,i_r}=1$.
\end{lemma}

\begin{proof}
Clearly, $h_r$
is  homogeneous of degree $r$, i.e.,
$$h_r(a_1,\ldots, a_q)=\sum_{1\le i_1\le i_2\le\ldots\le i_r\le q}
 c_{i_1,\ldots,i_r} a_{i_1}\cdot a_{i_2}\cdots \cdot a_{i_r}.$$
We first check that $c_{i_1,\ldots,i_r}\ne0$ is possible only for coefficients with
$1\le i_1<i_2<\ldots<i_r\le q$. For this consider indices $i_1,\ldots,i_r$ with
$|\{i_1,\ldots,i_r\}|=:n<r$. By changing the numbering of the variables
$a_1,\ldots,a_q$ (and  of rows and columns of $u$ in an appropriate
way), we  may assume that $\{i_1,\ldots,i_r\}=\{1,\ldots, n\}$. In this case,
$u^*\cdot diag(a_1,\ldots,a_n,0,\ldots,0)\cdot u$ has rank at most $n<r$.
Thus
$$0=h_r(a_1,\ldots,a_n,0,\ldots,0)=\sum_{1\le i_1\le i_2\le\ldots\le i_r\le n}
 c_{i_1,\ldots,i_r} a_{i_1}\cdot a_{i_2}\cdots \cdot a_{i_r}$$
for all $a_1,\ldots,a_n\in \mathbb R$. This yields $c_{i_1,\ldots,i_r}=0$ for $1\le i_1\le
i_2\le\ldots\le i_r\le n$. Therefore, for suitable coefficients,
$$h_r(a_1,\ldots, a_q)=\sum_{1\le i_1<i_2<\ldots<i_r\le q}
 c_{i_1,\ldots,i_r}a_{i_1}\cdot a_{i_2}\cdots \cdot a_{i_r}.$$

For the nonnegativity we again may restrict our attention to the coefficient
$c_{1,\ldots,r}$. In this case, with respect to the usual ordering of positive
definite matrices,
$$0\le \left(\begin{array}{cc} I_r&0\\0&0\end{array}\right)\le I_q
\quad\quad\text{and thus}\quad\quad
0\le u^*\left(\begin{array}{cc} I_r&0\\0&0\end{array}\right)u\le I_q.$$
As this inequality holds also for the upper left $r\times r$ block, 
$$c_{1,\ldots,r}=h_r(1,\ldots,1,0,\ldots,0)=
\Delta_r\left(u^*\left(\begin{array}{cc}
  I_r&0\\0&0\end{array}\right)u\right)\ge0.$$
Finally, as
$$\sum_{1\le i_1<i_2<\ldots<i_r\le q} c_{i_1,\ldots,i_r}= h_r(1,\ldots,1)=1,$$
the proof is complete.
\end{proof}

Let us keep the notation of Lemma \ref{pos-coeff}. We  compare
 $h_r$ with the homogeneous polynomial
\begin{equation}\label{C-r}
C_r(a_1,\ldots,a_q):=\frac{1}{{q\choose r}}
\sum_{1\le i_1<i_2<\ldots<i_r\le q}a_{i_1} a_{i_2}\cdots \cdot a_{i_r}>0 
\quad\quad(r=1,\ldots,q).
\end{equation}

\begin{lemma}\label{abs}
For all $a_1,\ldots,a_q>0$,
$$0<\frac{C_r(a_1,\ldots,a_q)}{h_r(a_1,\ldots,a_q)}\le \frac{1}{{q\choose r}} 
\sum_{1\le i_1<i_2<\ldots<i_r\le q}c_{i_1,\ldots,i_r}(u)^{-1},$$
where, depending on $u$, on both sides the value $\infty$ is possible.
\end{lemma}

\begin{proof} Positivity is clear by Lemma \ref{pos-coeff}. Moreover,
\begin{align}
C_r(a_1,\ldots,a_q)=&\frac{1}{{q\choose r}}
\sum_{1\le i_1<i_2<\ldots<i_r\le q}a_{i_1} a_{i_2}\cdots \cdot a_{i_r}
\notag\\
\le&\frac{\max_{1\le i_1<i_2<\ldots<i_r\le q}c_{i_1,\ldots,i_r}^{-1}}{{q\choose r}}
\sum_{1\le i_1<i_2<\ldots<i_r\le q}c_{i_1,\ldots,i_r}a_{i_1} a_{i_2}\cdots \cdot a_{i_r}
\notag
\end{align}
which immediately leads to the claim.
\end{proof}

We shall also need an  integrability result for principal minors of matrices 
$k\in K:=U(q,\mathbb F)$.
For this, we write $k$ as block matrix
$k=\left(\begin{array}{cc} k_r&*\\ *&k_{q-r}\end{array}\right)$ with
 $k_r\in M_r(\mathbb C)$ and $k_{q-r}\in M_{q-r}(\mathbb C)$.

 \begin{proposition}\label{int-finite} Keep the block matrix notation
   above. For $0\le \epsilon<1/2$,
 $$\int_K |\det k_r|^{-2\epsilon}\> dk<\infty.$$
\end{proposition}

\begin{proof}
The statement is clear for $r=q$.
 By Lemma \ref{glech-det} we may also assume 
$1\le r\le q/2$. In this case, we use the 
matrix ball 
$$B_r:= \{ w\in M_{r}(\mathbb F):\> w^*w\le I_r\}$$ 
as well as the ball 
 $B:= \{ y \in M_{1,r}(\mathbb F)\equiv \mathbb F^r: \> \|y\|_2^2\le 1\}$. We
 conclude from the truncation lemma 2.1 of \cite{R2} that
$$\int_K |\det k_r|^{-2\epsilon}\> dk=\frac{1}{\kappa_r}
\int_{B_r}|\det w|^{-2\epsilon}
\Delta(I_r-w^*w)^{(q-2r+1)\cdot d/2-1}\, dw$$
where $dw$ is the Lebesgue measure on the ball $B_r$ and
 $$\kappa_r:=\int_{B_r} \det(I_r-w^*w)^{(q-2r+1)\cdot d/2-1}\> dw\in ]0,\infty[.$$
Moreover,
by Lemma 3.7 and Corollary 3.8 of \cite{R1}, the mapping $P:B^r\to B_r$
with
\begin{equation}\label{trafo-P}
 P(y_1, \ldots, y_r):= \begin{pmatrix}y_1\\y_2(I_r-y_1^*y_1)^{1/2}\\ 
\vdots\\
  y_r(I_r-y_{r-1}^*y_{r-1})^{1/2}\cdots (I_r-y_{1}^*y_{1})^{1/2}\end{pmatrix}
\end{equation}
establishes a diffeomorphism such that the image of the measure 
$\det(I_r-w^*w)^{ (q-2r+1)\cdot d/2-1   }dw$
under  $P^{-1}$ is 
 $\,\prod_{j=1}^{r}(1-\|y_j\|_2^2)^{(q-r-j+1)\cdot d/2-1}dy_1\ldots
dy_r$. Moreover, we show in Lemma \ref{det-P} below that 
$$\det P(y_1, \ldots, y_r)=\det\left(\begin{array}{cc} y_1\\ \vdots\\y_r
\end{array}\right).$$
We thus conclude that
\begin{equation}
\int_K |\det k_r|^{-2\epsilon}\> dk=\frac{1}{\kappa_r} \int_B \ldots\int_B 
\left|\det\left(\begin{array}{cc} y_1\\ \vdots\\y_r
\end{array}\right)\right|^{-2\epsilon}
\prod_{j=1}^{r}(1-\|y_j\|_2^2)^{ (q-r-j+1)\cdot d/2-1   }dy_1\ldots
dy_r.
\end{equation}
This integral is finite for $\epsilon<1/2$, as 
one can use Fubini with an one-dimensional inner integral
 w.r.t.~the (1,1)-variable. After this inner integration, no further
 singularities appear  from the determinant-part in the remaining integral.
\end{proof}

\begin{lemma}\label{det-P} Keep the notations of the preceding proof. 
For all $y_1,\ldots,y_r\in B$,
$$\det P(y_1, \ldots, y_r)=\det\left(\begin{array}{cc} y_1\\ \vdots\\y_r
\end{array}\right).$$
\end{lemma}

\begin{proof}
Fix $y_1\in B$. The mapping  $y\mapsto y(I_r-y_1^*y_1)^{1/2}$ on $B$
has the following form: If $y$ is written as $y=ay_1+y^\perp$ in a unique way
with $a\in\mathbb F$ and $y^\perp\perp y_1$, then
  $y(I_r-y_1^*y_1)^{1/2}=\sqrt{1-\|y_1\|_2^2}\cdot a y_1+y^\perp$ 
(write $I_r-y_1^*y_1$ in an orthonormal basis with $y_1/\|y_1\|_2$ as a member!). Using
  linearity of the determinant in all lines, we thus conclude that
$$\det\begin{pmatrix}y_1\\ y_2(I_r-y_1^*y_1)^{1/2}\\ 
\vdots\\
  y_r(I_r-y_{r-1}^*y_{r-1})^{1/2}\cdots (I_r-y_{1}^*y_{1})^{1/2}\end{pmatrix}
=\det\begin{pmatrix}y_1\\ y_2\\ y_3(I_r-y_2^*y_2)^{1/2}\\ 
\vdots\\ y_r(I_r-y_{r-1}^*y_{r-1})^{1/2}\cdots
 (I_r-y_{2}^*y_{2})^{1/2}\end{pmatrix}.$$
The lemma now follows by an obvious induction.
\end{proof}

In the end of this section we present a technical result which will be central
below to derive that the covariance matrices $\Sigma$ of Sections 2 and 3 have
maximal rank in the non-degenerated cases.

\begin{lemma}\label{lin-unab}
Let $a_1,\ldots,a_q\in]0,\infty[$  such that at least two of these numbers
 are different.
 Consider the diagonal matrix $a=diag(a_1,\ldots,a_q)$.
 Then the functions 
$$ f_r:U(q,\mathbb F)\to \mathbb R, \quad k\mapsto\ln\Delta_r(k^*ak)$$  
with $r=1,\ldots,q-1$ and the constant function 1 on $U(q,\mathbb F)$  are 
linearly independent.
\end{lemma}

\begin{proof}
Without loss of generality we assume that 
 $a=diag(a_1,\ldots,a_1,a_{s+1},\ldots, a_q)$
where $a_1$ appears $s$-times with $1\le s\le q-1$, and where $a_1$ is
different from $a_{s+1},\ldots, a_q$.  

We consider permutation matrices 
$k_1,\ldots,k_q\in U(q,\mathbb F)$ as follows: $k_1$ is the identity. For
$l=2,\ldots,s$, the matrix $k_l^*ak_l$ appears from $a$ by interchanging the
entries 
$l-1$ and $s+1$, and finally, for
$l=s+1,\ldots,q$, the matrix $k_l^*ak_l$ appears from $a$ by interchanging the
entries $s$ and $l$.

We now form the $q\times q$-dimensional matrix $A$ with entries
$$A_{r,l}:=  \left\{ \begin{array}{cc}   \ln \Delta_r(k_l^*ak_l)
  &\text{for}\quad r=1,\ldots,q-1 \\
1 &\text{for}\quad r=q.\end{array}\right.$$
and check that $A$ is nonsingular which implies the claim.

For this we use the abbreviations $x:=\ln a_1$ 
and $y_l:=\ln a_{s+l}$ for $l=1,\ldots,q-s$.  We now subtract the first column of
$A$ from all other columns of $A$ and obtain the matrix
$$\left(\begin{array}{cccccccccccc}
x &\vline & y_1-x   &0&0&\cdots&0 &\vline&&&&\\
2x &\vline& y_1-x& y_1-x&&&    &\vline   &&&&\\
\vdots &\vline&\vdots&\vdots&\ddots&{\bf 0}& &\vline&&{\bf 0}&&\\
(s-1)x &\vline&y_1-x& y_1-x&\cdots& y_1-x&0 &\vline&0&\cdots&0&0\\
sx &\vline&y_1-x& y_1-x&\cdots& y_1-x& y_1-x &\vline& y_2-x&\cdots&y_{q-s-1}-x&y_{q-s}-x\\
\cline{1-12}\\
sx+y_1 &\vline&0&0&\cdots&0&0 &\vline&y_2-x&\cdots&y_{q-s-1}-x&y_{q-s}-x\\
sx+y_1+y_2 &\vline&0&0&\cdots&0&0 &\vline&0&\ddots&\vdots&y_{q-s}-x\\
\vdots &\vline&&&{\bf 0}&& &\vline&\vdots&{\bf 0}&\ddots&\vdots\\
sx+\sum_{l=1}^{s-q-1}y_l &\vline&0&0&\cdots&0&0 &\vline&0&\cdots&0&y_{q-s}-x\\
\cline{1-12}\\
1 &\vline&0&0&\cdots&0&0 &\vline&0&\cdots&0&0
\end{array}\right)$$

Using the block structure of $B$ and the triangular form of these blocks we
see that this matrix has the determinant
$$1\cdot (y_2-x)(y_3-x)\cdots (y_{q-s}-x)(y_1-x)^s\ne0.$$
Therefore, this matrix and thus $A$ are nonsingular as claimed.
\end{proof} 

\section{Oscillatory behavior of hypergeometric functions of type $A$ at the identity}

In this section we prove Propositions \ref{1.moment-function-a},
\ref{2.moment-function-a}, and  \ref{2.moment-function-a-pos-def}
 about the moment functions of first and second order
in Section 2. The most remarkable result in our eyes is the
 oscillatory behavior of hypergeometric functions of type $A$ at the identity character
in Proposition \ref{1.moment-function-a}(3) which is uniform for $t\in C_q^A$.

The proof of this fact relies on the results in Section 4 and on 
 the  following elementary observation:

\begin{lemma}\label{ln-est}
Let  $\epsilon\in]0,1]$, $M\ge1$ and $m\in\mathbb N$. Then there exists a constant $C=C(\epsilon,M,m)>0$
such that for all $z\in]0,M]$,
$$|\ln(z)|^m\le C\left(1+z^{-\epsilon}\right).$$
\end{lemma}

\begin{proof} Elementary calculus yields $|x^\epsilon\cdot\ln x|\le 1/(e\epsilon)$ for
$x\in]0,1]$ and
 the Euler number $e=2,71...$. This leads to the estimate for  $z\in]0,1]$.
The estimate is trivial for $z\in]1,M]$.
\end{proof}

\begin{proof}[Proof of Proposition \ref{1.moment-function-a}(3):]
Let $\lambda\in \mathbb R^q$, $t\in C_q^A$. Consider
 $$a:=(a_1,\ldots,a_q):= (e^{2t_1},\ldots,e^{2t_q}), \quad\text{and}\quad
a_t^2=e^{2\underline t}=diag(a_1,\ldots,a_q)\in
GL(q,\mathbb F).$$
 Then, by the Harish-Chandra integral representation 
 (\ref{int-rep-a}) and the integral representation of the moment functions in
(\ref{moment-function-a}),
we have to estimate
\begin{align}\label{diff1a}
R:=& R(\lambda,t):=
|\phi_{-i\rho-\lambda}^A(t)- e^{i\langle\lambda, m_{\bf 1}(t)\rangle}|
\\
=& \biggl|  \int_K exp\left(\frac{i}{2}\sum_{r=1}^q (\lambda_r-\lambda_{r+1})\cdot
\ln\Delta_r(k^*a_t^2k)\right)\> dk
\notag
\\
&\quad\quad
 -exp\left(\frac{i}{2} \int_K\sum_{r=1}^q 
(\lambda_r-\lambda_{r+1})\cdot\ln\Delta_r(k^*a_t^2k)\> dk\right)
\biggr|
\notag
 \end{align}
with the convention $\lambda_{q+1}:=0$. For $r=1,\ldots,q$, we now use the polynomial
$C_r$ from Eq.~(\ref{C-r})
and write the logarithms of the principal minors in (\ref{diff1a}) as 
\begin{equation}\label{log-prin-min}
\ln\Delta_r(k^*a_t^2k)=\ln C_r(a_1,\ldots,a_r) +\ln(H_r(k,a))
\quad\text{with}\quad
H_r(k,a):=\frac{\Delta_r(k^*a_t^2k)}{ C_r(a_1,\ldots,a_n)}.
 \end{equation}
With this notation and with $|e^{ix}|=1$ for $x\in\mathbb R$, we rewrite (\ref{diff1a}) as
\begin{align}\label{diff2a}
R=&\biggl|  \int_K exp\left(\frac{i}{2}\sum_{r=1}^q (\lambda_r-\lambda_{r+1})\cdot\ln(H_r(k,a))\right)\> dk
\notag
\\
&\quad\quad
 -exp\left(\frac{i}{2} \int_K\sum_{r=1}^q 
(\lambda_r-\lambda_{r+1})\cdot\ln(H_r(k,a))\> dk\right)\biggr|.
\end{align}
We now use the power series for both exponential functions where
 the terms of order 0 and 1 are equal. Hence, 
$R\le R_1+R_2$
for
$$R_1:=\int_K\biggl| 
exp\left(\frac{i}{2}\sum_{r=1}^q (\lambda_r-\lambda_{r+1})\cdot\ln(H_r(k,a))\right)
-\left(1+\frac{i}{2}\sum_{r=1}^q
(\lambda_r-\lambda_{r+1})\cdot\ln(H_r(k,a))\right)
\biggr|\> dk,$$
$$R_2:=\biggl|exp\left(\frac{i}{2} \int_K\sum_{r=1}^q
(\lambda_r-\lambda_{r+1})\cdot
\ln(H_r(k,a))\> dk\right)
\> -\> 1-\frac{i}{2}\int_K\sum_{r=1}^q (\lambda_r-\lambda_{r+1})\cdot\ln(H_r(k,a))\>
dk\biggr|.$$
Using the well-known elementary estimates $|\cos x-1|\le x^2/2$ and 
 $|\sin x-x|\le x^2/2$ for $x\in\mathbb R$, we obtain
$|e^{ix}-(1+ix)|\le x^2$ for $x\in\mathbb R$.
Therefore,
defining
$$A_m:= 2^{-m}
\int_K\biggl| \sum_{r=1}^q
(\lambda_r-\lambda_{r+1})\cdot\ln(H_r(k,a))\biggr|^m\>dk
\quad\quad(m=1,2),$$
we conclude that
$$R\le R_1+R_2\le A_2 +A_1^2.$$
In the following, let $D_1,D_2,\ldots$ suitable constants. As $A_1^2\le A_2$  by 
 Jensen's inequality, and as
$$A_2\le \|\lambda\|^2 \cdot D_1\cdot
\int_K\sum_{r=1}^q |\ln(H_r(k,a))|^2\>dk =:\|\lambda\|^2 \cdot B_2,$$
we obtain
$R\le B_2\cdot 2 \|\lambda\|^2 $.
To complete the proof, we must check that $B_2$, i.e., 
 the integrals
\begin{equation}\label{I-ma}
L_{r}:=\int_K |\ln(H_r(k,a))|^2\>dk
\end{equation}
remain bounded independent of 
$a_1,\ldots, a_q>0$ for $r=1,\ldots,q$.

For this fix  $r$.
 Lemma \ref{pos-coeff} in particular implies that for all $a_1,\ldots,a_q>0$,
$$\Delta_r(k^*a_t^2k)\le \sum_{1\le i_1<i_2<\ldots<i_r\le q}
a_{i_1}\cdot a_{i_2}\cdots \cdot a_{i_r} ={q\choose r}C_r(a_1,\ldots,a_q)$$
and $\Delta_r(k^*a_t^2k)>0$.
Hence, 
\begin{equation}\label{H-r}
0 <\frac{  \Delta_r(k^*a_t^2k) }{ C_r(a_1,\ldots,a_q)}=  H_r(k,a)  \le {q\choose r}.
\end{equation}
We conclude from (\ref{I-ma}), (\ref{H-r}) and Lemma \ref{ln-est} that for any 
$\epsilon\in]0,1[$ and suitable $D_2=D_2(\epsilon)$,
$$L_{r}\le D_2\int_K \left(1+ H_r(a_1,\ldots,a_q)^{-\epsilon}\right)\> dk.$$
Thus, by Lemma \ref{abs},
\begin{align}\label{suabs}
L_{r}\le& D_2 +D_3\int_K\left(  \sum_{1\le i_1<i_2<\ldots<i_r\le q} c_{i_1,\ldots,i_r}(k)^{-1}
  \right)^\epsilon\> dk
\notag\\
\le& D_2 +D_3\cdot {q\choose r}^\epsilon
 \sum_{1\le i_1<i_2<\ldots<i_r\le q}\int_K c_{i_1,\ldots,i_r}(k)^{-\epsilon}\> dk.
\end{align}
The right hand side of (\ref{suabs}) is independent of $a_1,\ldots,a_q$,
 and, by the definition of
 the $ c_{i_1,\ldots,i_r}(k)$ in Lemma \ref{pos-coeff},
 $\int_K c_{i_1,\ldots,i_r}(k)^{-\epsilon}\> dk$
is independent of $1\le i_1<i_2<\ldots<i_r\le q$.
 Therefore, it suffices to check that
\begin{equation}\label{final-est}
I_r:=\int_K c_{1,\ldots,r}(k)^{-\epsilon}\> dk =\int_K \Delta_r\left(k^*\left(\begin{array}{cc}
  I_r&0\\0&0\end{array}\right)k\right)^{-\epsilon}\> dk<\infty.
\end{equation}
For this, we write $k$ as block matrix
$k=\left(\begin{array}{cc} k_r&*\\ *&k_{q-r}\end{array}\right)$ with
 $k_r\in M_r(\mathbb C)$ and $k_{q-r}\in M_{q-r}(\mathbb C)$ and observe that
$$ \Delta_r\left(k^*\left(\begin{array}{cc}
  I_r&0\\0&0\end{array}\right)k\right) =\Delta_r\left(\begin{array}{cc}k_r^*k_r
&*\\ *&*\end{array}\right)=|\det k_r|^2.
$$
Therefore, (\ref{final-est}) follows from Proposition
\ref{int-finite}, which completes the proof of 
Proposition \ref{1.moment-function-a}(3).
\end{proof}

We now turn to the proof of the remaining parts of 
Proposition \ref{1.moment-function-a}.
Part (1)  is a
direct consequence of the law of large numbers \ref{lln-a}(1).
 Notice that in fact
the proof of this law of large numbers in Section 7 does not depend on
 Proposition \ref{1.moment-function-a}(1).
 Proposition \ref{1.moment-function-a}(2) is just part (3) of
 the following result:

\begin{lemma}\label{absch-moment-function-1-a}
For $r=1,\ldots,q$ let
$$s_r(t):= m_{(1,0,\ldots,0)}(t)+\cdots+  m_{(0,\ldots,0, 1,0,\ldots,0)}(t)
 \quad\quad\text{for} \quad
t\in C_q^A$$
be the sum of the first $r$ moment functions of first order.
Then:
\begin{enumerate}
\item[\rm{(1)}]For all $t\in C_q^A$,  $s_q(t)=t_1+t_2+\cdots+ t_q$.
\item[\rm{(2)}] There is a constant $C=C(q)$
 such that for all $r=1,\ldots,q$ and $t\in C_q^A$,
$$0\le t_1+t_2+\cdots+ t_r \> -\> s_r(t)\le C.$$
\item[\rm{(3)}]  There is a constant $C=C(q)$ such that for all $t\in C_q^A$
 $$\left\| t - m_{\bf 1}(t)\right\|\le C.$$
\end{enumerate}
\end{lemma}

\begin{proof}
By the integral representation  (\ref{moment-function-a}) of the moment
functions, we have
\begin{equation}\label{s-r-q}
s_r(t)=\frac{1}{2} \int_K \ln\Delta_r(k^*e^{2\underline t}k)\> dk
\quad\quad(r=1,\ldots,q).
\end{equation}
For $r=q$, this proves (1). Moreover, for $t\in C_q^A$ we have $t_1\ge
t_2\ge\ldots \ge t_q$. This and Lemma  \ref{pos-coeff} imply that for all
$k\in K$, 
\begin{equation}\label{est-delta-1}
\frac{1}{2} \ln\Delta_r(k^*e^{2\underline t}k)\le t_1+t_2+\cdots+ t_r.
\end{equation}
This and (\ref{s-r-q}) now lead to the first inequality of (2). 
For the second  inequality of (2), we use the notations 
of Lemmas \ref{pos-coeff} and  \ref{abs}. For $k\in K$ and 
$a_1:=e^{2t_1}\ge a_2:=e^{2t_2}\ge\ldots \ge a_q:=e^{2t_q}$ we obtain  from Lemma \ref{abs}
that
$$a_1\cdot a_2\cdots a_r\le {q \choose r} C_r(a_1,\ldots,a_q)\le
\Delta_r(k^*e^{2\underline t}k)\cdot M(k)$$
with 
$$M(k):=\max_{1\le i_1<\ldots<i_r\le q} c_{i_1,\ldots,i_r}(k)^{-1}$$
which may be equal to $\infty$ for some $k$. Therefore,
\begin{align}\label{est-delta-2}
t_1+t_2+\cdots+ t_r&=\frac{1}{2} \ln(a_1\cdot a_2\cdots a_r)=\frac{1}{2} \int_K 
\ln(a_1\cdot a_2\cdots a_r)\> dk
\\
&\le \frac{1}{2}\int_K \ln\Delta_r(k^*e^{2\underline t}k)\> dk + \int_K\ln M(k)\> dk
\notag\end{align}
with
 $$ \int_K\ln M(k)\> dk\le 
M:= \sum_{1\le i_1<\ldots<i_r\le q}\int_K\ln(c_{i_1,\ldots,i_r}(k)^{-1})\>
dk.$$
We claim that $M$ is finite. For this we observe that by the definition of the
 $c_{i_1,\ldots,i_r}(k)$ in Lemma  \ref{pos-coeff}, all integrals
in the sum in the definition of $M$ are equal. It is thus sufficient to
consider the summand with coefficient $c_{1,2,\ldots,r}(k)$.
 On the other hand,  we write
$k\in K$ as 
$$k=\left(\begin{array}{cc} k_1&*\\ *&*\end{array}\right)$$ with $r\times
  r$-block $k_1$ and observe that 
$$\int_K\ln(c_{1,2,\ldots,r}(k)^{-1})\> dk =-\int_K \ln\Delta_r\left(k^*
\left(\begin{array}{cc} I_r&0\\0&0\end{array}\right)k\right) \> dk
=-\int_K \ln\det(k_1^*k_1) \> dk, $$
which is finite as a consequence  of Lemma \ref{int-finite}. 
Therefore, $M$ is finite which proves (2). 

Finally, (3) is a consequence of (2).
\end{proof}

Lemma \ref{absch-moment-function-1-a}(3) implies that there exists a constant $C=C(q)>0$
such that for all $t\in C_q^A$, $\lambda\in \mathbb R^q$,
\begin{equation}
|e^{i\langle\lambda, t\rangle}-e^{i\langle\lambda, m_{\bf 1}(t)}\rangle|\le C\cdot\|\lambda\|.
\end{equation}
Therefore, we conclude from Proposition \ref{1.moment-function-a}(3):

\begin{corollary}\label{cor-absch}
There exists a constant $C=C(q)>0$ such that for all $t\in C_q^A$, $\lambda\in \mathbb R^q$,
$$\|\phi_{-i\rho-\lambda}(t)-e^{i\langle\lambda,t\rangle} \|\le
C\cdot(\|\lambda\|+\|\lambda\|^2)  .$$
\end{corollary}

We next turn to the proof of Proposition \ref{2.moment-function-a}.

\begin{proof}[Proof of Proposition \ref{2.moment-function-a}]
Let $t\in C_q^A$. Consider a non-trivial row vector $a=(a_1,\ldots,a_q)\in\mathbb
R^q\setminus\{0\} $ as well as the continuous functions
$$f_1(k):=\ln\Delta_1(k^*e^{2\underline t}k)\quad \text{and}
\quad
f_l(k):=\ln\Delta_l(k^* e^{2\underline t}  k)-\ln\Delta_{l-1}(k^*
e^{2\underline t}    k)
\quad(l=2,\cdots,q).$$
Then, by (\ref{moment-function-a}), (\ref{m1-vector}),  (\ref{m2-matrix}), and the Cauchy-Schwarz inequality,
\begin{equation}\label{CSU}
a\left( m_{\bf 2}(t)-m_{\bf 1}(t)^t m_{\bf 1}(t)\right)a^t
= \int_K \left(\sum_{l=1}^q a_l f_l(k)\right)^2\> dk 
-\left(\int_K \sum_{l=1}^q a_l f_l(k)\> dk\right)^2
\ge0.
\end{equation}
This shows part (1) of the proposition.
Moreover, for $t=c\cdot(1,\ldots,1)\in C_q^A$ with $c\in\mathbb R$, the
functions $f_l$ are constant on $K$ for all $l=1,\ldots,q$ which implies  
$\Sigma^2(t)=0$ and thus part (2).

For the proof of part (3) we notice that we have equality in (\ref{CSU}) 
 if and only if the function 
$$k\mapsto \sum_{l=1}^q a_l f_l(k)= (a_1-a_2)\ln\Delta_1(k^*e^{2\underline t} k)+\dots+ (a_{q-1}-a_q)\ln\Delta_{q-1}(k^*e^{2\underline t} k)
+a_q\ln\Delta_q(k^*e^{2\underline t}k)$$
 is constant on $K$. As $k\mapsto\ln\Delta_q(k^*e^{2\underline t}k)$ is constant on $K$, and as under the condition of (3),
the functions $k\mapsto\ln\Delta_r(k^*e^{2\underline t}k)$ ($r=1,\ldots,q-1$) and the constant function $1$ are linearly
 independent on $K$ by Lemma \ref{lin-unab}, the function 
 $k\mapsto \sum_{l=1}^q a_l f_l(k)$ is constant on $K$ precisely for
 $a_1=a_2=\ldots=a_q$.
 This proves that $\Sigma^2(t)$ has rank $q-1$ as claimed.

We next turn to part (4). We recall that Lemma \ref{pos-coeff} implies
$$2jt_q\le \ln\Delta_{j}(k^*e^{2\underline t}k)\le 2jt_1$$
 for $k\in K$, $t\in
C_q^A$, and $j=1,\ldots,q$. Therefore, by the integral representation 
 (\ref{moment-function-a}), 
\begin{align}|m_{j,l}(t)|\le& 
\frac{1}{4}\int_K \Bigl|\ln\Delta_{j}(k^*e^{2\underline t}k)-
\ln\Delta_{j-1}(k^*e^{2\underline t}k)\Bigr|\cdot
\Bigl|\ln\Delta_{l}(k^*e^{2\underline t}k)-
\ln\Delta_{l-1}(k^*e^{2\underline t}k)\Bigr|\> dk 
\notag\\
\le& \> \bigl((j-1)(t_1-t_q)+\max(|t_1|,|t_q|)\bigr)
\bigl((l-1)(t_1-t_q)+\max(|t_1|,|t_q|)\bigr)
\notag\end{align}
for $j,l=1,\ldots, q$ and $t\in C_q^A$. This implies part (4).

For the proof of part (5) we recall from the proof of 
Lemma \ref{absch-moment-function-1-a}(2) that for all $t\in C_q^A$ and $k\in K$,
$$0\le 2t_1- \ln\Delta_{1}(k^*e^{2\underline t}k) \le \ln M(k)$$
 with $M(k)\le\infty$ as
defined there for $r=1$. This leads to
\begin{align}
| (\ln\Delta_{1}(k^*e^{2\underline t}k))^2-4t_1^2| &=
(2t_1- \ln\Delta_{1}(k^*e^{2\underline t}k))\cdot |\ln\Delta_{1}(k^*e^{2\underline t}k) +2t_1|
\notag\\
&\le (2t_1- \ln\Delta_{1}(k^*e^{2\underline t}k) )\cdot (4|t_1|+\ln M(k))
\notag\\
&\le4|t_1| (2t_1-\ln\Delta_{1}(k^*e^{2\underline t}k) ) +(\ln M(k))^2.
\notag\end{align}
Thus,
$$\Bigl|\int_K (\ln\Delta_{1}(k^*e^{2\underline t}k))^2\> dk    -4t_1^2\Bigr|\le
4|t_1|\Bigl(2t_1-\int_K \ln\Delta_{1}(k^*e^{2\underline t}k)\> dk \Bigr) 
+\int_K(\ln M(k))^2\> dk.$$
As 
$$2t_1-\int_K \ln\Delta_{1}(k^*e^{2\underline t}k)\> dk $$
remains bounded for $t\in C_q^A$ by Lemma \ref{absch-moment-function-1-a},
 and as  $\int_K(\ln M(k))^2\> dk$ is finite as  a
 consequence  of Lemma \ref{int-finite} by the same arguments as in the end of
 the proof
 of Lemma \ref{absch-moment-function-1-a}, we see that for $t\in C_q^A$,
$$\Bigl|\int_K (\ln\Delta_{1}(k^*e^{2\underline t}k))^2\> dk -4t_1^2\Bigr|\le
 C(|t_1|+1)$$
which proves the first inequality of Proposition \ref{2.moment-function-a} (5).
For the proof of the second inequality, we again use the proof of 
Lemma \ref{absch-moment-function-1-a}(2) now for $r=q-1$. This and 
Lemma \ref{absch-moment-function-1-a}(1) lead to
$$0\le 
\ln\left(\frac{\Delta_{q}(k^*e^{2\underline
    t}k)}{\Delta_{q-1}(k^*e^{2\underline t}k)}
\right)
-t_q\le M(k)\le\infty$$
for $k\in K$, $t\in C_q^A$. This implies the second 
 inequality of Proposition \ref{2.moment-function-a} (5) in the same way as in
 the preceding case.
 \end{proof}

We finally turn to the proof of Proposition \ref{2.moment-function-a-pos-def}
which is closely related to  Proposition \ref{2.moment-function-a}.

\begin{proof}[Proof of Proposition \ref{2.moment-function-a-pos-def}]
 Let $\nu\in M^1(C_q^A)$ with finite second moments. Consider a 
 row vector $a=(a_1,\ldots,a_q)\in\mathbb R^q\setminus\{0\}$ 
as well as the continuous functions
$$f_1(k,t):=\ln\Delta_1(k^*e^{2\underline t}k)\quad \text{and}
\quad
f_l(k,t):=\ln\Delta_l(k^* e^{2\underline t}  k)-\ln\Delta_{l-1}(k^*
e^{2\underline t}    k)
\quad(l=2,\cdots,q)$$
on $K\times C_q^A$. Then, by the definition of $\Sigma^2(\nu)$, 
(\ref{moment-function-a}), (\ref{m1-vector}),  (\ref{m2-matrix}),
 and the Cauchy-Schwarz inequality,
\begin{align}\label{CSU-2}
a\Sigma^2(\nu)a^t&= a\left( m_{\bf 2}(\nu)-m_{\bf 1}(\nu)^t m_{\bf 1}(\nu)\right)a^t
\notag\\
&=\int_{C_q^A} \int_K \left(\sum_{l=1}^q a_l f_l(k,t)\right)^2\> dk\> d\nu(t) 
-\left(\int_{C_q^A} \int_K \sum_{l=1}^q a_l f_l(k,t)\> dk\> d\nu(t)\right)^2
\ge0
\end{align}
where equality holds if and only if the continuous function
$$h:(k,t)\mapsto \sum_{l=1}^q a_l f_l(k,t)= (a_1-a_2)\ln\Delta_1(k^*e^{2\underline t} k)+\dots+ (a_{q-1}-a_q)\ln\Delta_{q-1}(k^*e^{2\underline t} k)
+a_q\ln\Delta_q(k^*e^{2\underline t}k)$$
is constant on $K\times C_q^A$ $\nu\otimes \omega_K$-almost surely with the
uniform distribution $\omega_K$ on $K$. This just means that $h$ is constant
on $supp\>(\nu\otimes \omega_K)= (supp\>\nu)\times K$. 

Assume now that $\nu$ satisfies the conditions of part (1) of the proposition,
i.e., that 
$supp\>\nu\not\subset D_q:= \{c\cdot (1,\ldots,1): \> c\in\mathbb R\}\subset
C_q^A$, and that the orthogonal projection $\tau(\nu)\in M^1(D_q)$ of $\nu$
from $C_q^A$ onto $D_q$ is no point measure. Now choose $t\in
supp\>\nu\setminus D_s$. As $h(t,.)$ is constant on $K$, we conclude from the
proof of Proposition \ref{2.moment-function-a}(3) that
$a_1=a_2=\ldots=a_q$. Therefore,
 $h(k,t)=a_1\cdot \ln\Delta_{q}(k^*e^{2\underline t} k)= a_1(t_1+\ldots+t_q)$
is independent of $t$ for $t\in supp\>\tau(\nu)$ which leads to $a_1=0$.
This shows that under the conditions of part (3), $\Sigma^2(\nu)$ has full
rank as claimed.

Parts (2) and (3) also follow by the same arguments and those of
 Proposition \ref{2.moment-function-a}.
 \end{proof}

\section{Oscillatory behavior of hypergeometric functions of type $BC$ at the identity}

In this section we prove Propositions \ref{1.moment-function-bc},
\ref{2.moment-function-bc}, and  \ref{2.moment-function-bc-pos-def}
 about the moment functions on the Weyl chamber $C_q^B$.
The proofs are related to those for the A-case in Section 5.

We again start with  the
 oscillatory behavior of hypergeometric functions $\phi_\lambda^p$ 
of type B at the identity for $p> 2q-1$. For this we recall and modify 
two results
about principal minors and determinants from \cite{RV1}.
In our notation, Lemma 4.8 of \cite{RV1} is as follows:

\begin{lemma}\label{lemma-eins2} Let $t\in C_q^B$, $w\in B_q$, $u\in
  U(q,\mathbb F)$ and $r=1,\ldots,q$. Denote the ordered singular values of
  the $q\times q$-matrix $w$ by
 $1\ge \sigma_1(w)\ge\ldots\ge \sigma_q(w)\ge0$. Then
\[\frac{\Delta_r( g(t,u,w))}{\Delta_r(  g(t,u,0))}\in
 \big[(1-  \widetilde t\,\sigma_1(w))^{2r}, (1+  \widetilde t\,\sigma_1(w))^{2r}\big], \quad \text{with } \,
\widetilde t:=\min(t_1,1). \]
\end{lemma}

\begin{lemma}\label{lemma-eins4}
For each  $p> 2q-1$ there exists $\epsilon>0$ with
$$\int_{B_q} \Delta(I-w^*w)^{-\epsilon} dm_p(w) <\infty.$$
\end{lemma} 

\begin{proof} The proof is similar to that of Lemma 4.10 in \cite{RV1}. We
  consider the ball 
$$B:=\{y\in\mathbb F^q:\> \|y\|_2 <1\}$$
and the diffeomorphism 
\begin{equation}\label{trafo-P1}
 P: B^q\longrightarrow B_q, \quad (y_1, \ldots, y_q) \longmapsto
 \begin{pmatrix}y_1\\y_2(I_q-y_1^*y_1)^{1/2}\\ 
\vdots\\
  y_q(I_q-y_{q-1}^*y_{q-1})^{1/2}\cdots (I_q-y_{1}^*y_{1})^{1/2}\end{pmatrix};
\end{equation}
see Lemma  3.7 and Corollary 3.8 of \cite{R1} and Remark 2.6 of \cite{RV1}.
Using the transformation formula, 
these results also ensure that for some constant $\kappa>0$, 
$$\int_{B_q}\Delta(I-w^*w)^{-\epsilon} dm_p(w) =\frac{1}{\kappa}
\int_{B^q}\prod_{j=1}^{q}
(1-\|y_j\|_2^2)^ { d(p-q-j+1)/2-1 -\epsilon}
dy_1\ldots dy_q.$$
The second integral is clearly finite for $d(p-2q+1)/2-\epsilon>0$, i.e., for
$p-2q+1>2\epsilon/d$. This implies the claim.
\end{proof}

\begin{proof}[Proof of Proposition \ref{1.moment-function-bc}(3)]
Let $p> 2q-1$,  $\lambda\in \mathbb R^q$, and  $t\in C_q^B$.
We use the integral representations (\ref{phi-int-kurz-bc}) and
 (\ref{def-m1-bc})
 for the spherical functions and the associated moment functions $m_{\bf 1}$
and study 
\begin{align}\label{diff1}
R:=& R(\lambda,t):=
|\phi_{-i\rho-\lambda}^p(t)- e^{i\langle \lambda, m_{\bf 1}(t)\rangle}|
\\
=& \biggl|  
\int_{B_q} \int_{U(q,\mathbb F)} exp\Bigl(\frac{i}{2} 
\sum_{r=1}^q (\lambda_r-\lambda_{r+1})\ln\Delta_{r}( g(t,u,w))\Bigr)
\> du\> dm_p(w)
\notag \\
&\quad\quad
- exp\Bigl(\frac{i}{2}\int_{B_q} \int_{U(q,\mathbb F)} 
\sum_{r=1}^q (\lambda_r-\lambda_{r+1})\ln\Delta_{r}( g(t,u,w))
\> du\> dm_p(w) \Bigr)
\biggr|
\notag
 \end{align}
with the convention $\lambda_{q+1}=0$. 
We  use the homogeneous polynomials $C_r$ from (\ref{C-r}) for
$r=1,\ldots,q$ 
and write the logarithms of the principal minors in (\ref{diff1}) as 
\begin{equation}
\ln\Delta_r( g(t,u,w))=\ln C_r(\cosh^2 t_1,\ldots,\cosh^2 t_r) +\ln H_r(t,u,w) \end{equation}
with
\begin{equation}\label{hr-def}
H_r(t,u,w):=\frac{\Delta_r( g(t,u,w))}{ C_r(\cosh^2 t_1,\ldots,\cosh^2 t_r   )}.
 \end{equation}
Hence
\begin{align}\label{diff2}
R=&\biggl|\int_{B_q} \int_{U(q,\mathbb F)}   exp\Bigl(\frac{i}{2} \sum_{r=1}^q (\lambda_r-\lambda_{r+1})\ln H_r(t,u,w) 
\Bigr)
\>  du\> dm_p(w)
\notag \\
&\quad\quad
- exp\Bigl(\frac{i}{2}\int_{B_q} \int_{U(q,\mathbb F)} 
\sum_{r=1}^q (\lambda_r-\lambda_{r+1})\ln H_r(t,u,w)  \>  du\> dm_p(w)\Bigr)
\biggr|.
\notag
 \end{align}
As before,  the power series for both exponential functions lead to 
$R\le R_1+R_2$
for
\begin{align}
R_1&:=   \int_{B_q} \int_{U(q,\mathbb F)}    \biggl| 
exp\left(  \frac{i}{2} \sum_{r=1}^q  (\lambda_r-\lambda_{r+1}) \ln H_r(t,u,w))\right)
\notag \\
&\quad\quad\quad\quad\quad
-\left(1+ \frac{i}{2}\sum_{r=1}^q (\lambda_r-\lambda_{r+1})\cdot\ln H_r(t,u,w)\right)\biggr|\>   du\> dm_p(w),
\notag
 \end{align}
\begin{align}
R_2&:=\biggl|exp\left( \frac{i}{2} 
 \int_{B_q} \int_{U(q,\mathbb F)}
\sum_{r=1}^q (\lambda_r-\lambda_{r+1})\cdot\ln( H_r(t,u,w)\> dk\right)
\>
\notag \\
&\quad\quad\quad -\> 1-\frac{i}{2}   \int_{B_q} \int_{U(q,\mathbb F)} \sum_{r=1}^q (\lambda_r-\lambda_{r+1})\cdot\ln H_r(t,u,w))\>
  du\> dm_p(w)\biggr|.
\notag
 \end{align}
We now use
$|e^{ix}-(1+ix)|\le x^2$ for $x\in\mathbb R$ again and define
$$A_m:=\int_{B_q} \int_{U(q,\mathbb F)}
\biggl| \sum_{r=1}^q
(\lambda_r-\lambda_{r+1})\cdot\ln H_r(t,u,w)\biggr|^m\>  du\> dm_p(w)
\quad\quad(m=1,2).$$
Hence, by  Jensen's inequality,
$$R\le R_1+R_2\le A_2 +A_1^2\le 2A_2.$$
As
$$A_2\le \|\lambda\|^2 \cdot const.\cdot
 \int_{B_q} \int_{U(q,\mathbb F)}\sum_{r=1}^q |\ln H_r(t,u,w)|^2\>dk =:\|\lambda\|^2 \cdot B_2,$$
we obtain
$R\le B_2\cdot 2 \|\lambda\|^2 $.
To complete the proof, we  check that $B_2$, i.e., 
 the integrals
\begin{equation}\label{I-m}
L_{r}:= \int_{B_q} \int_{U(q,\mathbb F)} |\ln H_r(t,u,w)|^2\> du\> dm_p(w)
\end{equation}
remain bounded independent of 
$t$ for $r=1,\ldots,q$.
For this fix  $r$ and recall that by (\ref{hr-def}),
$$\ln H_r(t,u,w) =\ln\Delta_r( g(t,u,w)) - \ln  C_r(\cosh^2 t_1,\ldots,\cosh^2 t_r   ).$$
Moreover, by Lemma \ref{lemma-eins2} 
$$\ln\Delta_r( g(t,u,w)) -\ln\Delta_r( g(t,u,0))\in
 2r[\ln(1-  \sigma_1(w)), \ln(1+ \sigma_1(w))].$$
Thus,
\begin{align}\label{hr-absch}
|\ln H_r(t,u,w)|^2 &\le
 2\left|\ln\left(\frac{\Delta_r( g(t,u,0))}{ C_r(\cosh^2 t_1,\ldots,\cosh^2 t_r
   )}\right)\right|^2 \\&\quad\quad +8r^2(|\ln(1+ \sigma_1(w))|^2 +|\ln(1- \sigma_1(w))|^2).
\notag\end{align}
Moreover, by the definition of $B_q$,
\begin{equation}\label{hilf-61}
 \int_{B_q} \int_{U(q,\mathbb F)}|\ln(1+ \sigma_1(w))|^2 \> du\> dm_p(w)\le
  (\ln 2)^2.
\end{equation}
To handle the more critical term $|\ln(1- \sigma_1(w))|^2$, we use the 
elementary fact that for all $\epsilon>0$ and $x\in ]0,1[$,
 $|\ln x|\le x^{-\epsilon}$. As for $w\in B_q$,
 $1\ge\sigma_1(w)\ge\ldots\ge\sigma_q(w)\ge0$, we get
\begin{equation}\label{est-sigma1-delta}
\frac{1}{1-\sigma_1(w)}\le \frac{2}{1-\sigma_1(w)^2}\le \,2 \prod_{r=1}^q
\frac{1}{1-\sigma_r(w)^2}
= \frac{2}{\Delta(I-w^*w)}.
\end{equation}
We thus obtain that for all  $\epsilon>0$,
$$|\ln(1- \sigma_1(w))|^2 \le (1-\sigma_1(w))^{-2\epsilon}\le 2^{2\epsilon}\cdot
\Delta(I-w^*w)^{-2\epsilon}.$$
Thus, by Lemma \ref{lemma-eins4}, 
\begin{equation}\label{est-1-sigma-endlich}
 \int_{B_q} \int_{U(q,\mathbb F)}|\ln(1- \sigma_1(w))|^2 \> du\> dm_p(w)\le
const. \cdot\int_{B_q} 
 \Delta(I-w^*w)^{-2\epsilon}\, dm_p(w) <\infty.
\end{equation}
It is therefore sufficient to prove that
\begin{equation}
\int_{B_q} \int_{U(q,\mathbb F)}
\left|\ln\left(\frac{\Delta_r( g(t,u,0))}{ C_r(\cosh^2 t_1,\ldots,\cosh^2 t_r
   )}\right)\right|^2\> du\> dm_p(w)
\end{equation}
remains bounded independent of $t$. But this integral is equal to
\begin{equation}
\int_{U(q,\mathbb F)}\left|\ln\left(\frac{\Delta_r(u^* \cosh^2t u)}{ C_r(\cosh^2 t_1,\ldots,\cosh^2 t_r
   )}\right)\right|^2\> du,
\end{equation}
and this expression remains bounded independent of  $t$ by the proof of
Proposition \ref{1.moment-function-a}(3) in Section 5; 
see Eqs.~(\ref{log-prin-min}) and (\ref{I-ma})
 and the arguments after (\ref{I-ma}) there. This completes the proof.
\end{proof}

For the case $q=1$, Proposition \ref{1.moment-function-bc}(3)
 was proved in \cite{V2} by the same approach in the
 context of Jacobi functions; see also \cite{Z1}, \cite{Z2} for the context of
 Sturm-Liouville hypergroups. 

We now turn to the proof of the remaining parts of 
Proposition \ref{1.moment-function-bc}.
Part (1)  follows  from
 the LLN \ref{lln-bc}(1). Notice that
the proof of this LLN in Section 8 is independent from
 Proposition \ref{1.moment-function-bc}(1).
For the proof of part (2) we state the following result, which is 
related to estimates in the proof of Proposition
\ref{1.moment-function-bc}(1), and which reduces  estimates from the
BC-case to the A-case in Section 5.

\begin{lemma}\label{hilf-absch-moment-function-1-bc}
 For $r=1,\ldots,q$, $t\in C_q^B$, $u\in U(q,\mathbb F)$, and $w\in B_q$,
$$\bigl|\ln\Delta_r(g(t,u,w))  - \ln\Delta_r( u^*e^{2\underline t}u)\bigr|
\le
\ln 4+ 2r\cdot\max\bigl(|\ln(1-\sigma_1(w))|, \ln(1+\sigma_1(w))\bigr)$$
with
$$\int_{B_q}\max\bigl(|\ln(1-\sigma_1(w))|, \ln(1+\sigma_1(w))\bigr)\> dm_p(w)
<\infty.$$
\end{lemma}

\begin{proof} We conclude from Lemma \ref{lemma-eins2} that for
 $u\in U(q,\mathbb F)$ and $w\in B_q$,
$$\Delta_r(  g(t,u,0))(1- \sigma_1(w))^{2r}\le
\Delta_r( g(t,u,w))\le \Delta_r(  g(t,u,0))(1+ \sigma_1(w))^{2r}$$
and thus
\begin{equation}\label{help21}
\bigl|\ln\Delta_r(g(t,u,w))  - \ln\Delta_r( g(t,u,0))\bigr|
\le 2r\cdot\max\bigl(|\ln(1-\sigma_1(w))|, \ln(1+\sigma_1(w))\bigr).
\end{equation}
Moreover, as
$$\frac{1}{4 } u^*e^{2\underline t}u\le  u^* (\cosh\underline t)^2u\le
u^*e^{2\underline t}u,$$
we have
$$\left| \ln\Delta_r(g(t,u,0))-  \ln\Delta_r( u^*e^{2\underline t}u)\right|\le
\ln 4$$
for $t\in C_q^B$, $u\in U(q,\mathbb F)$. In combination with (\ref{help21}),
this leads to the first estimation of the lemma.
For the second statement, we first observe that 
$\int_{B_q}\ln(1+\sigma_1(w))\> dm_p(w)$ is obviously finite.
Moreover, $\int_{B_q}|\ln(1-\sigma_1(w))|\> dm_p(w)$
 is also finite as a consequence of
(\ref{est-1-sigma-endlich}). 
\end{proof}

Proposition \ref{1.moment-function-bc}(2)
 is  now part (2) of the following result:

\begin{lemma}\label{absch-moment-function-1-bc}
\begin{enumerate}
\item[\rm{(1)}] For $r=1,\ldots,q$ let
$$s_r^{BC}(t):= m_{(1,0,\ldots,0)}(t)+\cdots+  m_{(0,\ldots,0, 1,0,\ldots,0)}(t)
 \quad\quad\text{for} \quad
t\in C_q^B$$
be the sum of the first $r$ moment functions of first order.
Then there is a constant $C=C(q)$
 such that for all $r=1,\ldots,q$ and $t\in C_q^B$,
$$| t_1+t_2+\cdots+ t_r \> -\> s_r^{BC}(t)| \le C.$$
\item[\rm{(2)}]  There is a constant $C=C(q)$ such that for all $t\in C_q^A$
 $$\left\| t - m_{\bf 1}(t)\right\|\le C.$$
\end{enumerate}
\end{lemma}

\begin{proof}
Let $t\in C_q^B$.
By the integral representation  (\ref{def-m1-bc}) of the moment
functions, we have
\begin{equation}\label{s-r-q-bc}
s_r^{BC}(t)=\frac{1}{2} \int_{B_q} \int_{U(q,\mathbb F)} 
\ln\Delta_r(g(t,u,w))\>  du\> dm_p(w)
\quad\quad(r=1,\ldots,q).
\end{equation}
We thus obtain from Lemma \ref{hilf-absch-moment-function-1-bc} that
for all $t\in C_q^B$ and $r=1,\ldots,q$,
$$ \left| s_r^{BC}(t)- \frac{1}{2}\int_{U(q,\mathbb F)} \ln\Delta_r( u^*e^{2\underline t}u  )\>
du\right|       \le C$$
for some constant $C>0$.
Therefore, in the notation of Lemma  \ref{absch-moment-function-1-a},
$$ \left| s_r^{BC}(t)- s_r(t)\right|       \le C
\quad\quad(t\in C_q^B,\>\> r=1,\ldots,q).$$
 Lemma
 \ref{absch-moment-function-1-a}(2)
now implies that for all $t\in C_q^B$ and $r=1,\ldots,q$,
$$|s_r(t)- (t_1+\ldots t_r)|\le \tilde C$$
for some constant $\tilde C$. This proves  part (1).
 Part (2) is a consequence of part (1).
\end{proof}

\begin{remark}
We conjecture that in part (1) of the preceding lemma the stronger result
\begin{equation}\label{abschae-bc-unklar}
0\le t_1+\ldots t_r -s_r^{BC}(t)\le C \quad\quad (r=1,\ldots,q, t\in C_q^B)
\end{equation}
holds which would correspond to Lemma \ref{absch-moment-function-1-a}(2)
 in the A-case. 

In fact, this could be easily derived
from the attempting matrix inequality
$$(\cosh\underline t+\sinh \underline t\cdot w)(\cosh\underline t+\sinh
\underline t\cdot w)^*\le e^{2\underline t} \quad\quad (t\in C_q^B,  w\in B_q).$$
 Unfortunately, this matrix inequality is not correct. Take for instance
 $q=2$, $t=(t_1,0)$ with $t_1$ large, and 
$w=\begin{pmatrix} 0&1 \\1&0 \end{pmatrix}$. Therefore, a proof of
 (\ref{abschae-bc-unklar}) would be more involved than in the A-case above.
\end{remark}

We next turn to the proof of Proposition \ref{2.moment-function-bc}.

\begin{proof}[Proof of Proposition \ref{2.moment-function-bc}]
Fix $t\in C_q^B$. Consider a non-trivial row vector
 $a=(a_1,\ldots,a_q)\in\mathbb
R^q\setminus\{0\} $ and the continuous functions
$$f_1(u,w):=\ln\Delta_1(g(t,u,w))\quad \text{and}
\quad
f_l(u,w):=\ln\Delta_l(g(t,u,w))-\ln\Delta_{l-1}(g(t,u,w))
\quad(l=2,\cdots,q)$$
on $U(q,\mathbb F)\times B_q$. Then, by (\ref{def-m1-bc}), (\ref{m1-vector-bc}),  (\ref{m2-matrix-bc}), and the Cauchy-Schwarz inequality,
\begin{align}\label{CSU-bc}
a&\left( m_{\bf 2}(t)-m_{\bf 1}(t)^t m_{\bf 1}(t)\right)a^t\\
&= \int_{B_q}\int_{U(q,\mathbb F)} \left(\sum_{l=1}^q a_l f_l(u,w)\right)^2\>
du\> dm_p(w) 
-\left(\int_{B_q}\int_{U(q,\mathbb F)} \sum_{l=1}^q a_l f_l(u,w)\> du\> dm_p(w) \right)^2
\ge0.
\notag\end{align}
This shows part (1) of the proposition.

For the proof of (2), use that for
$t=0\in C_q^B$, $m_{\bf 1}(0)=0$ and $m_{\bf 2}(0)=0$ which yields 
$\Sigma^2(t)=0$.

For the proof of part (3) we take $t\in C_q^B$ with $t\ne0$ and 
notice that we have equality in (\ref{CSU-bc}) 
 if and only if the function 
\begin{align}
(u,w)\mapsto &\sum_{l=1}^q a_l f_l(u,w)\notag\\
&= (a_1-a_2)\ln\Delta_1(g(t,u,w))+\dots+ (a_{q-1}-a_q)\ln\Delta_{q-1}(g(t,u,w))
+a_q\ln\Delta_q(g(t,u,w))
\notag\end{align}
 is constant on $U(q,\mathbb F)\times B_q$. 
Assume now that this is the case.

We now first consider the case where $t\in C_q^B$ does not have the form
$t=c(1,\ldots,1)$ with some $c>0$. In this case we put
 $w=I_q\in B_q$ with
$g(t,u,I_q)=u^*e^{2\underline t}u$. Therefore, 
 $$u\mapsto\ln\Delta_q(u^*e^{2\underline t}u)$$
 is constant on $U(q,\mathbb F)$,
and by our assumption and Lemma \ref{lin-unab},
the functions $u\mapsto\ln\Delta_r(u^*e^{2\underline t}u)$ ($r=1,\ldots,q-1$)
 and the constant function $1$ are linearly
 independent on  $U(q,\mathbb F)$. Consequently, as the function 
 $$(u,w)\mapsto \sum_{l=1}^q a_l f_l(u,w)$$ is constant on
 $U(q,\mathbb F)\times B_q$, we have
 $a_1=a_2=\ldots=a_q$. On the other hand, 
$$\ln\Delta_q(g(t,u,w))=
\ln|\Delta(\cosh \underline t+\sinh \underline t \cdot w)|^2$$
 is not constant
in $w\in B_q$ for $t\ne0$, which proves  $a_q=0$. This shows that
 $\Sigma^2(t)$ is positive definite for  $t\in B_q$ not having the form $ c(1,\ldots,1)$.
Finally, if $t$ has the form $t=c(1,\ldots,1)$ with some $c>0$, we may choose
$w=\begin{pmatrix} 1&0\\ 0&0 \end{pmatrix}\in B_q$ (with $1\in\mathbb R$).
Then  $g(t,u,w)=u^*D(t)u$ with some diagonal matrix $D(t)$ where not all
diagonal entries are equal. As above, Lemms \ref{lin-unab}
 yields in this case that  $a_1=a_2=\ldots=a_q$ and the proof can be completed
 in the same way as in the preceding case.

We next turn to part (4) of the proposition.
We recall that Lemma \ref{pos-coeff} implies
$$2jt_q\le \ln\Delta_{j}(u^*e^{2\underline t}u)\le 2jt_1$$
 for $u\in U(q,\mathbb F)$, $t\in
C_q^B$, and $j=1,\ldots,q$. Therefore, by the integral representation 
 (\ref{def-m1-bc}) of the moment functions and by 
Lemma \ref{hilf-absch-moment-function-1-bc}
\begin{align}|m_{j,l}(t)|\le& 
\frac{1}{4}\int_{B_q} \int_{U(q,\mathbb F)} 
\Bigl|\ln\Delta_{j}(g(t,u,w))-\ln\Delta_{j-1}(g(t,u,w))
\Bigr|
\cdot\notag\\
&\quad\quad\quad\quad\quad\cdot
\Bigl|\ln\Delta_{l}(g(t,u,w))-\ln\Delta_{l-1}(g(t,u,w))
\Bigr|\> du\> dm_p(w)\notag\\
\le& 
C+\frac{1}{4}\int_{U(q,\mathbb F)} \Bigl|\ln\Delta_{j}(u^*e^{2\underline t}u)-
\ln\Delta_{j-1}(u^*e^{2\underline t}u)\Bigr|\cdot
\Bigl|\ln\Delta_{l}(u^*e^{2\underline t}u)-
\ln\Delta_{l-1}(u^*e^{2\underline t}u)\Bigr|\> du 
\notag\\
\le& \> C(1+t_1)
\notag\end{align}
for $j,l=1,\ldots, q$, $t\in C_q^B$, and some constant $C>0$.
On the other hand, by the definition of $g(t,u,w)$, the functions 
 $\ln \Delta_{l}(g(t,u,w))$
are analytic at $t=0$ with $\ln \Delta_{l}(g(0,u,w))=0$. Therefore, 
$m_{j,l}(t)=O(t_1^2)$ for small $t\in C_q^B$. We thus obtain that 
$|m_{j,l}(t)|\le t_1^2$ for all $t\in C_q^B$ and $j,l$ with  some constant
$C>0$ as claimed in part (4).

For the proof of part (5),
we recall from the proof of 
Lemma \ref{absch-moment-function-1-bc}(2) above that for all $t\in C_q^B$, $u\in
U(q,\mathbb F)$, and $w\in B_q$,
$$ |2t_1- \ln\Delta_{1}(g(t,u,w))| \le \ln M(u,w)$$
 with 
some expression $M(u,w)\le\infty$ satisfying 
$\int_{U(q,\mathbb F)}\int_{B_q}  \ln M(u,w)\> du\> dm_p(w)<\infty$.
 This leads to
\begin{align}
| (\ln\Delta_{1}(g(t,u,w)))^2-4t_1^2| &=
(2t_1- \ln\Delta_{1}(g(t,u,w)))\cdot |\ln\Delta_{1}(g(t,u,w)) +2t_1|
\notag\\
&\le (2t_1- \ln\Delta_{1}(g(t,u,w)) )\cdot (4|t_1|+\ln M(u,w))
\notag\\
&\le4|t_1| (2t_1-\ln\Delta_{1}(g(t,u,w)) ) +(\ln M(u,w))^2.
\notag\end{align}
Therefore,
\begin{align}
\Bigl|\int_{U(q,\mathbb F)}&\int_{B_q} (\ln\Delta_{1}(g(t,u,w)))^2\> du\> dm_p(w)  
  -4t_1^2\Bigr|\le\notag\\
&\le4|t_1|\Bigl(2t_1-\int_{U(q,\mathbb F)}\int_{B_q}
 \ln\Delta_{1}(g(t,u,w))\> du\> dm_p(w) \Bigr) 
+\int_{U(q,\mathbb F)}\int_{B_q}(\ln M(u,w))^2\> du\> dm_p(w).
\notag\end{align}
Therefore,
as 
$$\Bigl(2t_1-\int_{U(q,\mathbb F)}\int_{B_q} \ln\Delta_{1}(g(t,u,w))\>
 du\> dm_p(w)  \Bigr)$$
remains bounded for $t\in C_q^B$ by Lemma \ref{absch-moment-function-1-bc},
 and as also  
$$\int_{U(q,\mathbb F)}\int_{B_q} (\ln M(u,w))^2\>  du\> dm_p(w)<\infty$$
  by the  arguments  of
 the proof
 of Lemma \ref{absch-moment-function-1-a}, we conclude that for $t\in C_q^B$,
$$\Bigl|\int_{U(q,\mathbb F)}\int_{B_q} 
 (\ln\Delta_{1}(g(t,u,w)))^2\>   du\> dm_p(w) -4t_1^2\Bigr|\le
 C(|t_1|+1)$$
as claimed in Proposition \ref{2.moment-function-bc}(5).
\end{proof}

We finally turn to the proof of Proposition \ref{2.moment-function-bc-pos-def}
which is closely related to  Proposition \ref{2.moment-function-bc}(3).

\begin{proof}[Proof of Proposition \ref{2.moment-function-bc-pos-def}]
 Let $\nu\in M^1(C_q^B)$ with finite second moments and $\nu\ne\delta_0$.
 Consider a 
 row vector $a=(a_1,\ldots,a_q)\in\mathbb R^q\setminus\{0\}$ 
as well as the continuous functions
$$f_1(u,w,t):=\ln\Delta_1(g(t,u,w))\quad \text{and}
\quad
f_l(k,t):=\ln\Delta_l(g(t,u,w))-\ln\Delta_{l-1}(g(t,u,w))
\quad(l=2,\cdots,q)$$
on $U(q,\mathbb F)\times B_q\times C_q^B$.
 By the definition of $\Sigma^2(\nu)$, 
 (\ref{def-m1-bc}),  (\ref{m1-vector-bc}), (\ref{m2-matrix-bc}),
 and the Cauchy-Schwarz inequality,
\begin{align}\label{CSU-2-bc}
a\Sigma^2(\nu)a^t&= a\left( m_{\bf 2}(\nu)-m_{\bf 1}(\nu)^t m_{\bf 1}(\nu)\right)a^t
\notag\\
&=\int_{C_q^B} \int_{U(q,\mathbb F)}\int_{B_q}  \left(\sum_{l=1}^q a_l
f_l(k,t)\right)^2\> dm_p(w) \> du\> d\nu(t) \notag\\
&\quad\quad\quad-\left(\int_{C_q^B} \int_{U(q,\mathbb F)}\int_{B_q}
 \sum_{l=1}^q a_l f_l(k,t)\> dm_p(w) \> du\> d\nu(t)\right)^2
\quad \ge\quad 0
\end{align}
where equality holds if and only if the continuous function
$$h:(u,w,t)\mapsto \sum_{l=1}^q a_l f_l(u,w,t)$$
is constant on $K\times C_q^A$ almost surely
 w.r.t.~$ \omega_{U(q,\mathbb F)}\times m_p\times \nu$-almost surely. This however, is not the case by the
  proof of Proposition \ref{2.moment-function-bc}(3).
\end{proof}

\section{Proof of the stochastic limit theorems in the case A}

In this section we prove the strong law of large numbers \ref{lln-a} and the CLT
\ref{clt-a}  for $K:=U(q,\mathbb F)$-biinvariant random walks 
$(S_k)_{k\ge0}$ on $G:=GL(q,\mathbb F)$ 
associated with some
$\nu\in M^1(C_q^A)$: 

We first turn to the CLT \ref{clt-a}. Besides the results of Section 5 we need
the  following estimate which follows immediately from the
 integral representation (\ref{int-rep-a}) for the functions $\phi_\lambda^A$.

\begin{lemma}\label{absch-abl-a}
For all $t\in C_q^A$, $\lambda\in \mathbb R^q$, and  $l\in\mathbb N_0^q$,
$$\left|\frac{\partial^{|l|}}{\partial\lambda^l} 
\phi_{-i\rho-\lambda}^A(t)\right|\le m_l(t).$$
\end{lemma}

Let $m\in\mathbb N_0$. We
say that $\nu\in  M^1(C_q^A)$  admits finite  $m$-th modified moments if in the notation of
Section 2, 
$$m_{(m,0,\ldots,0)}, m_{(0,m,0,\ldots,0)},\ldots, m_{(0,\ldots,0,m)}\in L^1(C_q^A,\nu).$$
It follows  from the integral representation 
(\ref{moment-function-a}) of the moment function and H\"older's inequality that
in this case all moment functions of order at most $m$ are $\nu$-integrable.
Moreover, this moment condition implies a corresponding differentiability of
the spherical Fourier transform of  $\nu$:

\begin{lemma}\label{differentiable-a}
Let $m\in\mathbb N_0$ and $\nu\in M^1(C_q^A)$ 
 with finite $m$-th moments.
 Then the spherical Fourier transform
$$\tilde\nu:\mathbb R^q\to\mathbb C,\quad
 \lambda\mapsto\int_{C_q^A} \phi_{-i\rho-\lambda}^A(t)\> d\nu(t)$$
is $m$-times continuously partially differentiable,
 and for all $l\in\mathbb N_0^n$ with $|l|\le m$,
\begin{equation}\label{tauschallg-a}
\frac{\partial^{|l|}}{\partial\lambda^l}\tilde\nu(\lambda)
=\int_{C_q^A}\frac{\partial^{|l|}}{\partial\lambda^l} 
\phi_{-i\rho-\lambda}^A(t)\> d\nu(t).
\end{equation}
In particular,
\begin{equation}\label{tausch-a}
\frac{\partial^{|l|}}{\partial\lambda^l}\tilde\nu(0)=(-i)^{|l|} \int_{C_q^A}
 m_l(t)\> d\nu(t).
\end{equation}
\end{lemma}

\begin{proof} We proceed by induction: The case $m=0$ is trivial,
 and for $m\to m+1$
 we observe that by our assumption all moments of lower order
 exist, i.e., (\ref{tauschallg-a}) is available for
all $|l|\le m$. It  follows from Lemma \ref{absch-abl-a}
 and a well-known result about parameter integrals that
a further partial derivative and  integration can be interchanged.
 Finally, (\ref{tausch-a}) follows from
 (\ref{tauschallg-a}) and (\ref{moment-function-a}).
 Continuity of the derivatives
is also clear by Lemma \ref{absch-abl-a}.
\end{proof}

We now turn to the proof of the CLT:

 \begin{proof}[Proof of  Theorem \ref{clt-a}]
 Let $\nu\in M^1(C_q^A)$ be a 
 probability measure with finite second modified moments.
Let $(X_k)_{k\ge1}$ be i.i.d.~$G$-valued 
random variables with the associated $K$-biinvariant distribution 
$\nu_G\in M^1(G)$ and $S_k:=X_1\cdot X_2\cdots X_k$ as in Section 2.
We consider the canonical projection 
$$(\tilde S_k:=\ln\sigma_{sing}(S_k))_{k\ge0}$$
 of this random
walk from $G$ to $G//K\simeq C_q$ as in Section 2.

 Let $\lambda\in\mathbb R^q$.
As the functions $\phi_{-i\rho-\lambda}^A$ are bounded on $C_q^A$
 (by the integral representation (\ref{int-rep-a}))
and multiplicative w.r.t.~double coset convolutions of  measures on $C_q^A$,
 we have
$$E(\phi_{-i\rho-\lambda/\sqrt k}^A(\tilde S_k))=
 \int_{C_q^A} \phi_{-i\rho-\lambda/\sqrt k}^A(t)\> d\nu^{(k)}(t)=
\left(\int_{C_q^A} \phi_{-i\rho-\lambda/\sqrt k}^A(t)\> d\nu(t)\right)^k =
\tilde\nu(\lambda/\sqrt k)^k.$$
We now use Taylor's formula, Lemma \ref{differentiable-a}, and
$$m_{\bf 2}(\nu):=\int_G m_{\bf 2}(g)\> d\nu(g)=\Sigma^2(\nu)+m_{\bf 1}(\nu)^tm_{\bf 1}(\nu)$$
and obtain
\begin{align}
exp(&i\langle\lambda, m_{\bf 1}(\nu)\rangle\cdot  \sqrt k)\cdot
 E(\phi_{-i\rho-\lambda/\sqrt k}^A(\tilde S_k))\quad =\quad
\left(exp(i\langle\lambda, m_{\bf 1}(\nu)\rangle/ \sqrt k)\cdot
\tilde\nu(\lambda/\sqrt k)\right)^k
\\
=&\left(\left[1+\frac{i\langle \lambda, m_{\bf 1}(\nu)\rangle}{ \sqrt k}- 
\frac{\langle\lambda, m_{\bf 1}(\nu)\rangle^2}{ 2 k} +o(\frac{1}{k})\right]\cdot
\left[1-\frac{i\langle \lambda, m_{\bf 1}(\nu)\rangle}{ \sqrt k}-
\frac{\lambda m_{\bf 2}(\nu)\lambda^t}{ 2 k} +o(\frac{1}{k})\right]
\right)^k
\notag \\
=&\biggl(\left[1+\frac{i\langle  \lambda, m_{\bf 1}(\nu)\rangle}{ \sqrt k}- 
\frac{(\langle \lambda, m_{\bf 1}(\nu)\rangle^2}{ 2 k} +o(\frac{1}{k})\right]\cdot
\notag \\
&\quad\quad \times\left[1-\frac{i \langle\lambda, m_{\bf 1}(\nu)\rangle}{ \sqrt k}-
\frac{\lambda (\Sigma^2(\nu)+ m_{\bf 1}(\nu)^t m_{\bf 1}(\nu))\lambda^t}{ 2 k} +o(\frac{1}{k})\right]
\biggr)^k
\notag \\
=&\left( 1-\frac{\lambda \Sigma^2(\nu)\lambda^t}{ 2 k} +o(\frac{1}{k})\right)^k.
\notag
\end{align}
Therefore,
$$\lim_{k\to\infty} exp(i\langle\lambda, m_{\bf 1}(\nu)\rangle \sqrt k)\cdot
 E(\phi_{-i\rho-\lambda/\sqrt k}^A(\tilde S_k))
= exp(-\lambda  \Sigma^2(\nu)\lambda^t/ 2 ).$$
Moreover, by Proposition \ref{1.moment-function-a}(3)
$$\lim_{k\to\infty} E\left( \phi_{-i\rho-\lambda/\sqrt k}^A(\tilde S_k)-
exp(-i\langle\lambda, m_{\bf 1}(\tilde S_k)\rangle /\sqrt k)\right) =0.$$
Therefore,
$$\lim_{k\to\infty} exp(-i\langle \lambda, ( m_{\bf 1}(\tilde S_k)- k\cdot m_{\bf 1}(\nu) )\rangle /\sqrt k)=
 exp(-\lambda  \Sigma^2(\nu)\lambda^t/ 2 )$$
for all $\lambda\in\mathbb R^q$. Levy's continuity theorem for the classical $q$-dimensional
 Fourier transform now implies that
$( m_{\bf 1}(\tilde S_k)-k\cdot m_{\bf 1}(\nu))/\sqrt k$ tends in distribution to $N(0, \Sigma^2(\nu))$.
By the estimate of Lemma \ref{absch-moment-function-1-a}(3), this shows 
that $( \tilde S_k-k\cdot m_{\bf 1}(\nu))/\sqrt k$ tends in
distribution to $N(0, \Sigma^2(\nu))$ as claimed.
\end{proof}

The oscillatory behavior of the $\phi_\lambda^A$ in
Proposition \ref{1.moment-function-a}(3) can be used to derive 
 a  Berry-Esseen-type estimate with
 the order $ O(k^{-1/3})$ of convergence.
 As the details are technical and  similar to 
 the corresponding rank-one-case in Theorem 4.2 of \cite{V3}, we  omit it.
We also mention that 
Proposition \ref{1.moment-function-a}(3) can be  used to derive further CLTs
e.g. with stable distributions as limits or a Lindeberg-Feller CLT. The details of proof
then would be also very similar to the classical  sums of iid random
variables.

We next turn to the strong law of large numbers \ref{lln-a}.

 \begin{proof}[Proof of  Theorem \ref{lln-a}]
We first prove part (2) and consider some $\nu\in M^1(C_q^A)$ having second
moments. Let $\epsilon>1/2$. We employ the strong law of large numbers 7.3.21 in  \cite{BH}
for the random walk $(\tilde S_k)_{k\ge0}$ on the double coset hypergroup 
$G//K\simeq C_q^A$ with the constants $r_k:=k^{-2\epsilon}$ there which satisfy
$\sum_{k=1}^\infty r_k<\infty$. For all $l=1,\ldots,q$, we now apply this
result to one-dimensional  moment functions $m_{(0,\ldots,0,1,0,\ldots,0)}$
and $m_{(0,\ldots,0,2,0,\ldots,0)}$ of first and second order 
with the nontrivial entry in the position $l$. 
 The integral representation (\ref{moment-function-a}) and Jensen's inequality
 ensure that 
$$m_{(0,\ldots,0,1,0,\ldots,0)}^2\le m_{(0,\ldots,0,2,0,\ldots,0)}
 \quad\quad\text{on}\quad C_q^A,$$
i.e., condition (MF2) for Theorem  7.3.21 in  \cite{BH} holds. 
We conclude from this theorem that for all  $l=1,\ldots,q$ and vectors of the
form $(0,\ldots,0,1,0,\ldots,0)$ with $1$ at position $l$,
$$k^{-\epsilon}\cdot\bigl( m_{(0,\ldots,0,1,0,\ldots,0)}(\tilde S_k) - k\int_{C_q^A}
 m_{(0,\ldots,0,1,0,\ldots,0)}(t)\> d\nu(t)\bigr) $$
tends to $0$ a.s. for $k\to\infty$.
 In other words, $k^{-\epsilon}\cdot\bigl(m_{\bf 1}(\tilde S_k)
-k\cdot m_{\bf 1}(\nu))$ tends to $0$ a.s..
Proposition \ref{1.moment-function-a}(2) finally implies that
 $k^{-\epsilon}\cdot\bigl(\tilde S_k -k\cdot m_{\bf 1}(\nu))$ tends to $0$ a.s. as
claimed.

 Part (1) follows in the same way from Theorem 7.3.24 in  \cite{BH} with the
 constant $\lambda=1$  there.
\end{proof}

\section{Proof of the stochastic limit theorems in the case BC}

In this section we prove the LLN \ref{lln-bc}
 and the CLT \ref{clt-bc} in the BC-case. Based on the technical results of
 Section 6, the proofs are very similar to those in Section 7.
We therefore skip  details.

We first turn to the CLT \ref{clt-bc}. Besides  Section 6 we need
the  following  immediate consequence of the
 integral representation (\ref{phi-int-kurz-bc}) for  $\phi_\lambda^p$.

\begin{lemma}\label{absch-abl-bc}
For all $t\in C_q^B$, $\lambda\in\mathbb R^q$, and  $l\in\mathbb N_0^q$,
$$\left|\frac{\partial^{|l|}}{\partial\lambda^l} 
\phi_{-i\rho-\lambda}^p(t)\right|\le m_l(t).$$
\end{lemma}

Let $m\in\mathbb N_0$. We
say that  $\nu\in  M^1(C_q^A)$ admits finite  $m$-th modified moments if in the notation of
Section 3, 
$$m_{(m,0,\ldots,0)}, m_{(0,m,0,\ldots,0)},\ldots, m_{(0,\ldots,0,m)}\in L^1(C_q^B,\nu).$$
By the integral representation 
(\ref{def-m1-bc})  and by H\"older's inequality,
in this case all moment functions of order at most $m$ are $\nu$-integrable.
Moreover, this moment condition leads to a corresponding differentiability of
the spherical Fourier transform of  $\nu$ as in Lemma \ref{differentiable-a}.
We omit the proof:

\begin{lemma}\label{differentiable-bc}
Let $m\in\mathbb N_0$ and $\nu\in M^1(C_q^B)$ 
 with finite $m$-th moments.
 Then the spherical Fourier transform
$$\tilde\nu:\mathbb R^q\to\mathbb C,\quad
 \lambda\mapsto\int_{C_q^B} \phi_{-i\rho-\lambda}^p(t)\> d\nu(t)$$
is $m$-times continuously partially differentiable,
 and for  $l\in\mathbb N_0^n$ with $|l|\le m$,
\begin{equation}
\frac{\partial^{|l|}}{\partial\lambda^l}\tilde\nu(\lambda)
=\int_{C_q^B}\frac{\partial^{|l|}}{\partial\lambda^l} 
\phi_{-i\rho-\lambda}^p(t)\> d\nu(t).
\end{equation}
In particular,
\begin{equation}
\frac{\partial^{|l|}}{\partial\lambda^l}\tilde\nu(0)=(-i)^{|l|} \int_{C_q^B}
 m_l(t)\> d\nu(t).
\end{equation}
\end{lemma}

We now turn to the proof of the CLT:

 \begin{proof}[Proof of  Theorem \ref{clt-bc}]
 Let $\nu\in M^1(C_q^B)$ be a 
 probability measure with finite second modified moments,
 $p\in[ 2q-1,\infty[$, and  $d=1,2,4$.
As described in Section 3 we consider the associated time-homogeneous 
random walk $(\tilde S_k)_{k\ge0}$ on $C_q^B$. Then, as described there, 
 the distributions of
$\tilde S_k$ are given as the convolution powers $\nu^{(k)}$ w.r.t.  $*_p$.
With this observation in mind, we can just use the results of Section 6
and Lemmas \ref{absch-abl-bc} and \ref{differentiable-bc} instead of the
results of Section 5 and Lemmas \ref{absch-abl-a} and \ref{differentiable-a},
respectively in order to complete
 the proof in the same way as for the CLT \ref{clt-a}.
\end{proof}

Finally, the strong law of large numbers  \ref{lln-bc} can be proved 
by the same methods as the strong law  \ref{lln-a} in Section 7 by using the 
integral representation (\ref{def-m1-bc}) of the moment functions instead of
(\ref{moment-function-a}).

\end{document}